\documentclass[12pt]{article}
\usepackage[T1]{fontenc}
\usepackage[utf8]{inputenc}
\usepackage{authblk}
\usepackage[misc]{ifsym}
\usepackage{dsfont}
\usepackage{mathrsfs}
\usepackage{amsmath,amssymb}
\usepackage{amsfonts}

\usepackage{amsthm}
\usepackage{graphicx}
\usepackage{subfigure}
\usepackage{xcolor}

\usepackage{bbm}
\usepackage{mathrsfs}
\usepackage{txfonts}
\usepackage{color}
\usepackage{pict2e}
\usepackage{float}
\usepackage{palatino,epsfig,latexsym}
\usepackage{epsf,xy,epic,amscd}
\usepackage{lineno}

\usepackage{bibentry}
\usepackage[colorlinks,linkcolor=blue,anchorcolor=blue,citecolor=blue,]{hyperref}

\theoremstyle{plain}
\newtheorem{theorem}{Theorem}[section]
\newtheorem{lemma}[theorem]{Lemma}

\newtheorem{proposition}[theorem]{Proposition}

\newtheorem{Bounded Diameter Lemma}[theorem]{Bounded Diameter Lemma}
\theoremstyle{definition}
\newtheorem{definition}[theorem]{Definition}

\newtheorem{remark}[theorem]{Remark}

\newcommand{\Hmm}[1]{\leavevmode{\marginpar{\tiny%
$\hbox to 0mm{\hspace*{-0.5mm}$\leftarrow$\hss}%
\vcenter{\vrule depth 0.1mm height 0.1mm width \the\marginparwidth}%
\hbox to
0mm{\hss$\rightarrow$\hspace*{-0.5mm}}$\\\relax\raggedright #1}}}

\hfuzz=\maxdimen
\DeclareFixedFont{\Acknowledgment}{OT1}{cmr}{bx}{n}{14pt}
\begin{document}
\numberwithin{equation}{section}

 \title{The Dirichlet problem for Monge-Amp\`ere equation for $(n-1)$-PSH functions on Hermitian manifolds}
\author{Rirong Yuan\thanks{
		 School of Mathematics, \,South China University of Technology, \,Guangzhou 510641, \,China \\ 	\indent 	\indent 
		Email address:\,yuanrr@scut.edu.cn 
}
}
\date{}

\maketitle

\begin{abstract}

	We solve the Dirichlet problem for Monge-Amp\`ere equation for $(n-1)$-PSH functions possibly with degenerate right-hand side, through deriving a quantitative version of boundary estimate under the assumption of  $(n-1)$-PSH subsolutions.  
	  In addition, we confirm the subsolution assumption
	  on a product of a closed balanced manifold with a compact Riemann surface with boundary.


\end{abstract}



 \section{Introduction}
 
 A basic work in K\"ahler geometry is Yau's \cite{Yau78} 
 proof of Calabi's conjecture.
 In \cite{Gauduchon77}  Gauduchon showed that
 every closed Hermitian manifold $(M,J,\omega)$ of complex dimension $n\geq 2$
 admits a unique (up to rescaling)  Gauduchon metric
in the conformal class $[\omega]$, and he furthermore
 conjectured in \cite{Gauduchon84} that the Calabi-Yau theorem 
 is also true for Gauduchon metric. 
 On  complex surfaces and astheno-K\"ahler manifolds 
 introduced by Jost-Yau \cite{Jost1993Yau},
 the Gauduchon conjecture was solved by Cherrier \cite{Cherrier1987} and Tosatti-Weinkove \cite{TW19}, respectively.
 The Gauduchon conjecture in higher dimensions 
 amounts to solving a Monge-Amp\`ere type equation 
 \begin{equation}
 	\label{MA-n-1}
 	\begin{aligned}
 		\left(\tilde{\chi}+\frac{1}{n-1}\left(\Delta u\omega-\sqrt{-1}\partial \overline{\partial}u\right)+Z\right)^n=\phi \omega^n
 \end{aligned}  \end{equation} 
 where  
 $$\tilde{\chi}=\frac{1}{(n-1)!}* \omega_0^{n-1}, \quad 
 Z = \frac{1}{(n-1)!}*\mathfrak{Re}\left(\sqrt{-1}\partial u\wedge \overline{\partial}(\omega^{n-2})\right),$$   see  \cite{Popovici2015,TW19}. Here $*$ is the Hodge star operator with respect to $\omega$ and $\omega_0$ is a Gauduchon metric.  
 Following \cite{Harvey2011Lawson}, also \cite{TW17}, we call $u$ an $(n-1)$-plurisubharmonic  ($(n-1)$-PSH in short) function if 
 \begin{equation}
 	\label{gamma-cone}
 	\begin{aligned}
 		\tilde{\chi}+\frac{1}{n-1}\left(\Delta u\omega-\sqrt{-1}\partial \overline{\partial}u\right)+Z\geq0
 		\mbox{ in }  M. \nonumber 
 	\end{aligned}
 \end{equation}
Furthermore, we call $u$ a strictly $(n-1)$-PSH function if 
 $$\tilde{\chi}+\frac{1}{n-1}\left(\Delta u\omega-\sqrt{-1}\partial \overline{\partial}u\right)+Z>0 \mbox{ in } M.$$
 Geometrically,   
  the strictly $(n-1)$-PSH solution of
  \eqref{MA-n-1} allows one to construct 
 a Gauduchon metric, say $\Omega_u$, 
 with prescribed volume form
 \begin{equation}
 	\label{Prescribed-Volume-form1}
 	\begin{aligned}
 		 \Omega_u^{n}= \tilde{\phi}\omega^n  
 	\end{aligned}
 \end{equation}
and $\Omega_u^{n-1}=\omega_0^{n-1}+ \sqrt{-1}\partial \overline{\partial}u \wedge \omega^{n-2} + \mathfrak{Re}(\sqrt{-1}\partial u \wedge\overline{\partial}\omega^{n-2}),$
 where and hereafter 
 \begin{equation}
 	\label{tildephi-1}
 	\begin{aligned}
 		\tilde{\phi}=\phi^{{1}/{(n-1)}}. 
 	\end{aligned}
 \end{equation}
 Recently, 
 Sz\'ekelyhidi-Tosatti-Weinkove \cite{STW17} proved Gauduchon's conjecture for $n\geq3$, thereby extending the Calabi-Yau theorem to Gauduchon metric.

 \vspace{1mm}
 This paper is devoted to extending Sz\'ekelyhidi-Tosatti-Weinkove's results to complex manifolds with boundary.
 To this end, we assume that $(M,J,\omega)$ is a compact Hermitian manifold of complex dimension $n\geq 2$  with  boundary $\partial M$, $\bar M:=M\cup\partial M$, 
 where $J$ is the complex structure, $\omega= \sqrt{-1}  g_{i\bar j} dz^i\wedge d\bar z^j$ denotes the K\"ahler form. We consider the Dirichlet problem with prescribed boundary value condition
 \begin{equation}
 	\label{bdy-value-2}
 	\begin{aligned}
 		u=\varphi \mbox{ }\mbox{ on } {\partial M}.
 	\end{aligned}
 \end{equation}
 
Before stating our results, we summarize some notion.

\begin{definition}
	 A strictly $(n-1)$-PSH function $\underline{u}\in C^{2}(\bar M)$ is called a \textit{subsolution} of the Dirichlet problem \eqref{MA-n-1} and \eqref{bdy-value-2}, if 
 \begin{equation}
 	\label{subsolution-MA-n-1}
 	\begin{aligned}
 		\left(\tilde{\chi}+\frac{1}{n-1}\left(\Delta \underline{u}\omega-\sqrt{-1}\partial \overline{\partial}\underline{u}\right)+\underline{Z}\right)^n\geq \,& \phi \omega^n \,& \mbox{ in } M, \\
 		\underline{u}=\,& \varphi \,& \mbox{ on } \partial M. 
 	\end{aligned}
 \end{equation} 
We say a strictly $(n-1)$-PSH function $\underline{u}\in C^2(\bar M)$ is a \textit{strict subsolution} of the Dirichlet problem \eqref{MA-n-1} and \eqref{bdy-value-2}, if
\begin{equation}
	\label{strictly-subsolution-MA-n-1}
	\begin{aligned}
		\left(\tilde{\chi}+\frac{1}{n-1}\left(\Delta \underline{u}\omega-\sqrt{-1}\partial \overline{\partial}\underline{u}\right)+\underline{Z}\right)^n \geq\,& (\phi+\delta)\omega^n \,& \mbox{ in }  M \\
		\underline{u}=\,& \varphi \,& \mbox{ on } \partial M  
	\end{aligned}
\end{equation} 
 for a positive constant $\delta$.
Here $\underline{Z} = \frac{1}{(n-1)!}*\mathfrak{Re}(\sqrt{-1}\partial \underline{u}\wedge \overline{\partial}(\omega^{n-2}))$. 

\end{definition}


  Our main results are stated as follows.
 \begin{theorem}
 	\label{thm0-n-1-yuan3}
 	Let $(M,J,\omega)$ be a compact Hermitian manifold with smooth boundary.  
 	Assume the given data $\varphi\in C^{\infty}(\partial M)$ and 
 	$0<\phi\in C^\infty(\bar M)$ support an $(n-1)$-PSH  subsolution $\underline{u}\in C^{2,1}(\bar M)$ subject to \eqref{subsolution-MA-n-1}.
 	Then the Dirichlet problem \eqref{MA-n-1} and \eqref{bdy-value-2} 
 	is uniquely solvable in the class of smooth strictly $(n-1)$-PSH functions. 
 	 
 \end{theorem}

We say the boundary is 
 \textit{mean pseudoconcave} if
 \begin{equation}
 	\label{mean-pseudoconcave1}
 	\begin{aligned}
 		-(\kappa_1+\cdots+\kappa_{n-1})\geq 0 \quad \mbox{ on } \partial M 
 	\end{aligned}
 \end{equation} 
where $\kappa_1,\cdots, \kappa_{n-1}$ are the eigenvalues of  Levi form $L_{\partial M}$
with respect to 
$\omega'=\omega|_{T_{\partial M}\cap JT_{\partial M}}$.
When the boundary is mean pseudoconcave we can solve the equation
with degenerate right-hand side $\phi\geq0$.

 \begin{theorem}
 	\label{thm2-n-1-yuan3}
 	Let $(M,J,\omega)$ be a compact Hermitian manifold with smooth mean pseudoconcave  boundary. Let $\varphi\in C^{2,1}(\partial M)$ and $0\leq\phi^{{1}/{n}}\in C^{1,1}(\bar M)$.  Suppose in addition that there exists a $C^{2,1}$-smooth $(n-1)$-PSH strict subsolution 
 	satisfying \eqref{strictly-subsolution-MA-n-1}.
 	Then the Dirichlet problem \eqref{MA-n-1} and \eqref{bdy-value-2}
 	admits a (weak) $(n-1)$-PSH solution $u\in C^{1,\alpha}(\bar M)$, $\forall 0<\alpha<1$,
 	with $\Delta u \in L^{\infty}(\bar M)$.
 \end{theorem}
 
 As shown in the above theorem, we obtain a $C^{1,\alpha}$-smooth $(n-1)$-PSH weak solution for \eqref{MA-n-1}
 with right-hand side of the form $\phi=w^n$ for some $0\leq w\in C^{1,1}(\bar M)$.
 However, it is fairly restrictive even in the case $n=2$.
 In some problems arising from complex geometry and analysis, 
 it would be interesting 
 to build up 
 $\Omega_u^{n}= \tilde{\phi}\omega^n$ 
 with $\partial\overline{\partial}\Omega_u^{n-1}=0$ 
in the weak sense,  given a  nonnegative
 $C^{1,1}$-smooth
 function $\tilde{\phi}$.

 
 We make progress on this subject under the assumption that
 \begin{equation}
 	\label{keykey-phi}
 	\begin{aligned}
 		\frac{\partial \tilde{\phi}}{\partial \nu}=0 \mbox{ }\mbox{ at } \left\{p\in\partial M: \tilde{\phi}(p)=0\right\}
 	\end{aligned}
 \end{equation}
 where $\nu$ is the unit inner normal vector along the boundary. Such a condition is automatically satisfied if the data $(M,\tilde{\phi})$ is \textit{extensible} in the sense that
 \begin{itemize} 
 	\item $(M,\omega)$ is a complex submanifold of an open Hermitian manifold $(M',\omega)$ of complex dimension $n$, 
 	$\bar M\subsetneq M'$; and 
 	\item 
 	$\tilde{\phi}$  can be $C^{1,1}$-smoothly extended to a nonnegative function on $M'$.
 \end{itemize}

 \begin{theorem}
 	\label{them1-n-1-phi}
 	Let $(M,J,\omega)$  be a compact Hermitian manifold with smooth \textit{mean pseudoconcave} boundary.  
 	Let 	$\varphi\in C^{2,1}(\partial M)$ and $\tilde{\phi}=\phi^{1/(n-1)}$ be as in \eqref{tildephi-1} for $\phi\geq0$. 
 Assume the data  $\tilde{\phi}\in C^{1,1}(\bar M)$ satisfies  \eqref{keykey-phi}  and suppose a $C^{2,1}$ $(n-1)$-PSH strict subsolution subject to \eqref{strictly-subsolution-MA-n-1}. Then 
 $$\Omega_u^n=\tilde{\phi}\omega^n$$
 	possesses an $(n-1)$-PSH weak solution $u\in C^{1,\alpha}(\bar M)$ with $\forall 0<\alpha<1$, $u|_{\partial M}=\varphi$,
 	$\Delta u \in L^{\infty}(\bar M)$.
 	
 \end{theorem}
 
 The existence of 
 subsolution 
 is imposed as an important ingredient to study the Dirichlet problem, subsequent to the works of 
  \cite{Guan1993Spruck,Guan1998The}
   concerning Mong-Amp\`ere equations on domains. It would be worthwhile to note that the subsolution assumption is a necessary condition for the solvability of Dirichlet problem.
 
Let  $(X, J_X, \omega_X)$ be a closed balanced manifold (introduced by Michelsohn \cite{Michelsohn1982Acta}), let $(S, J_S, \omega_S)$ be a compact Riemann surface  with boundary $\partial S$.
We confirm the subsolution assumption 
on the standard product 
 \begin{equation}
 	\label{product-1}
 	\begin{aligned}
 		(M,J,\omega)=(X\times S,J,\pi_1^*\omega_X+\pi_2^*\omega_S),
 \end{aligned}\end{equation}
when  the boundary data $\varphi$ can be extended to a $C^{2,1}$-smooth functions on $\bar M$, still denoted by
$\varphi$, satisfying
\begin{equation}
	\label{bdy-n-1}
	\begin{aligned}
		\tilde{\chi}+\frac{1}{n-1}\left(\Delta \varphi\omega-\sqrt{-1}\partial \overline{\partial}\varphi\right)+ Z[\varphi] +t\pi_1^*\omega_X>0  
	 \mbox{ for some } t\gg1.
	\end{aligned}
\end{equation}
Here  
 $\pi_{1}$, $\pi_2$ denote the nature projections:
 \begin{equation}
 	\begin{aligned}
 		\,&  \pi_1: X\times S\rightarrow X, 
 		\,&  \pi_2: X\times S\rightarrow S.  \nonumber
 	\end{aligned}
 \end{equation} 
 The construction starts with the solution of
 \begin{equation}  
 	\label{possion-def}
 	\begin{aligned}
 		\,&  \Delta_S h =1 \mbox{ } \mbox{ in } S, \,& h=0 \mbox{ } \mbox{ on } \partial S.  
 	\end{aligned}
 \end{equation} 
 According to Lemma \ref{obser-subsolution} below, the obstruction to constructing subsolutions  
 vanishes whenever  
 $\omega_X$ is
 balanced. 
We can verify 
 $$*(\sqrt{-1}\partial\overline{\partial}h\wedge\omega^{n-2})
 = (n-2)! (\Delta h\omega-\sqrt{-1}\partial\overline{\partial}h)
 =(n-2)!\pi_1^*\omega_X.$$
The $(n-1)$-PSH subsolution is given by
 \begin{equation}   \begin{aligned}
 		\underline{u}=\varphi+t h \quad \mbox{for } t\gg1. \nonumber
 \end{aligned}\end{equation}

 As a consequence, we can solve Dirichlet problem  \eqref{MA-n-1} and \eqref{bdy-value-2} on such products. In addition, combining with Propositions \ref{proposition-quar-yuan2} and \ref{mix-Leviflat}, we obtain some more delicate results
  with less regularity assumption.
 
  \begin{theorem}
 	\label{thm2-volumeform}
 	
 	Let $0<\beta<1$, and let $(M,J,\omega)=(X\times S,J,\pi_1^*\omega_X+\pi_2^*\omega_S)$
 	be as above with 
 	$\partial S\in C^{2,\beta}$. For any $0\leq\tilde{\phi}\in C^{1,1}(\bar M)$ satisfying \eqref{keykey-phi},  
 	the Dirichlet problem 
 	\begin{equation}
 		\label{VolForm-homo1}
 		\begin{aligned}
 			\Omega_u^n=\tilde{\phi}\omega^n 
 			\mbox{ } \mbox{ in } M, 
 			\quad u=0 \mbox{ }  \mbox{ on } \partial M 
 		\end{aligned}
 	\end{equation}
 	admits a $C^{1,\alpha}$ 
 	$(n-1)$-PSH solution with  $\forall 0<\alpha<1$, 
 	$\Delta u\in {L^\infty(\bar M)}$  in the weak sense.

 \end{theorem}
 
 \begin{theorem}
 	\label{thm2-volumeform-2}
 	
 	Let $(M,J,\omega)=(X\times S,J,\pi_1^*\omega_X+\pi_2^*\omega_S)$
 	be as above with 
 	$\partial S\in C^{2,\beta}$ for some $0<\beta<1$. Given a  $0<\tilde{\phi}\in  C^{2}(\bar M)$,
  there is a unique $C^{2,\alpha}$-smooth $(n-1)$-PSH function for some $0<\alpha\leq\beta$ solving  \eqref{VolForm-homo1}.

 \end{theorem}

 Finally, we make some remarks:


  \vspace{1mm}
 $\mathrm{{(1)}}$  Our results 
 are valid when replacing $\tilde{\chi}=\frac{1}{(n-1)!}*\omega_0^{n-1}$ by  more general real $(1,1)$-forms $\tilde{\chi}$.

  \vspace{1mm}
 $\mathrm{{(2)}}$
 It is still open to derive gradient estimate directly for Monge-Amp\`ere equation for $(n-1)$-PSH functions. 
 Our strategy is to derive  a quantitative version of 
 boundary estimate 
 \begin{equation}
 	\begin{aligned}
 		\sup_{\partial M} \Delta u \leq C\left(1+\sup_M |\nabla u|^2\right). \nonumber
 	\end{aligned}
 \end{equation}
 This is precisely the estimate we prove in Theorem \ref{thm-21-yuan}. In the case of $Z=0$ such a boundary estimate was derived by the author in 
 \cite{yuan-V}.
 It would be worthwhile to note that when $Z\neq 0$ the proof of quantitative boundary estimate
 is much more complicated and fairly 
 difficult, due to the gradient terms from $Z$. It requires some new ideas and insights. 

\vspace{1mm}
$\mathrm{{(3)}}$  We prove in Theorem \ref{them1-n-1-phi} the existence of $C^{1,\alpha}$-smooth solution to 
$\Omega_u^n=\tilde{\phi}\omega^n$ with the assumption  \eqref{keykey-phi} but without restriction to the dimension. 
That is different from
the works of Guan \cite{GuanP1997Duke} and Guan-Li \cite{Guan1997Li} on certain degenerate real Monge-Amp\`ere type equations, in which the assumptions  were only completely confirmed in dimensions $2$ and $3$.

 \vspace{1mm}
 $\mathrm{{(4)}}$   In Theorems \ref{thm2-volumeform}-\ref{thm2-volumeform-2} we find a significant phenomenon on weakening regularity assumptions on boundary. 
 For the Dirichlet problem with homogeneous boundary data
 on $M=X\times S$,
 the regularity assumption on boundary can be weakened to $C^{2,\beta}$; while
 such $C^{2,\beta}$ regularity assumption 
 is impossible even for Dirichlet problem of
 nondegenerate real Monge-Amp\`ere equation on certain bounded domains $\Omega\subset \mathbb{R}^2$,
 as shown by Wang \cite{WangXujia-1996}, the optimal regularity assumptions on the boundary and boundary data are both $C^3$-smooth.
 More results 
 with less regularity assumption are established in  Section  \ref{sec5}.

  \vspace{1mm}
 $\mathrm{{(5)}}$ When $\omega_0$ is a balanced metric and $\omega$ is astheno-K\"ahler, one obtains  a $\mathrm{d}$-closed $(n-1,n-1)$-form
 \[
 \omega_0^{n-1}+\sqrt{-1}\partial \overline{\partial}u \wedge \omega^{n-2} + 2\mathfrak{Re}(\sqrt{-1}\partial u \wedge\overline{\partial}\omega^{n-2}),\]
 the Form-type Calabi-Yau equation introduced in \cite{FuWangWuFormtype2010} thus falls into a Monge-Amp\`ere type equations analogous to \eqref{MA-n-1}. By the same argument, we obtain a balanced metric $\Psi_u$ with prescribed volume form 
 \[\Psi_u^n=\tilde{\phi}\omega^n\]
 with the boundary value \eqref{bdy-value-2}; similarly, 
 we can construct the subsolutions on the products as we constructed in \eqref{product-1}. It was shown by Latorre-Ugarte
 \cite{Latorre2017Ugarte} that there are closed complex manifolds $X$ admitting a non-K\"ahler balanced metric
 $\omega_{X,0}$ 
 and a non-K\"ahler metirc $\omega_X$  simultaneously being astheno-K\"ahler and Gauduchon.
 On that product $X\times S$ one obtains a balanced metric
 $\omega_0=\pi_1^*\omega_{X,0}+\pi_2^*\omega_{S}$ and an astheno-K\"ahler metric
 $\omega=\pi_1^*\omega_X+\pi_2^*\omega_S$.
 As a result, we can 
 solve the Dirichlet problem for Form-type Calabi-Yau equation on such products. 

 \vspace{1.5mm}
 The paper is organized as follows. 
 In Section \ref{sketchproof} we sketch the proof of main theorems.
 In Section \ref{sec2} we summarize some lemmas.
 In Sections \ref{sec3} and \ref{sec4} we derive the quantitative boundary estimate, which is the main part of this paper.
 In Section \ref{sec5} we also use our estimate to derive more delicate results for Dirichlet problem with less regularity assumptions.
 In Section \ref{sec6} we briefly extend our main results to more general equations.
 In Appendix \ref{appendix1} we summarize and give the proof of a quantitative lemma proposed in \cite{yuan2017}, which is a crucial ingredient for quantitative boundary estimate.

 \section{Sketch of the proof}
 \label{sketchproof}
 
  It is mysterious to derive gradient estimate for
  equation \eqref{MA-n-1} directly, as done by Hanani \cite{Hanani1996},  B{\l}ocki \cite{Blocki09gradient} and Guan-Li \cite{Guan2010Li} for complex Monge-Amp\`ere equation. 
In \cite[Section 3]{STW17},   Sz\'ekelyhidi-Tosatti-Weinkove
derived gradient estimate 
for \eqref{MA-n-1} on closed Hermitian manifolds via a blow-up argument, 
establishing  the second order estimate of the form
\begin{equation}
	\label{sec-estimate-quar1}
	\begin{aligned}
		\sup_{M}\Delta u\leq C \left(1+\sup_{M} |\nabla u|^2\right) 
	\end{aligned}
\end{equation}
left open by Tosatti-Weinkove \cite{TW19}. 
Such a blow-up argument using \eqref{sec-estimate-quar1} appeared in previous works  as done by
Chen \cite{Chen}, complemented by \cite{Boucksom2012,Phong-Sturm2010},
for Dirichlet problem of complex 
Monge-Amp\`ere equation, 
and by Dinew-Ko{\l}odziej \cite{Dinew2017Kolo} based on Hou-Ma-Wu's second estimate 
\cite{HouMaWu2010}  for complex $k$-Hessian equations on
closed K\"ahler manifolds. 
An extensive extension 
 was obtained by Sz\`ekelyhidi  \cite{Gabor} on closed manifolds, and   by the author \cite{yuan-V}
for the Dirichlet problem.

 When $\partial M\neq\emptyset$,  
Sz\'ekelyhidi-Tosatti-Weinkove's estimate  
 yields the following:
 \begin{theorem}
 	[\cite{STW17}]
 	\label{STW-estimate1}
 	Let $u\in C^{4}(M)\cap C^2(\bar M)$ be a $(n-1)$-PSH solution to  \eqref{MA-n-1}, then
 	\begin{equation}
 		\label{STW-estimate}
 		\begin{aligned}
 			\sup_{M} \Delta u
 			\leq C\left(1+ \sup_{M}|\nabla u|^{2} +\sup_{\partial M}|\Delta u|\right)
 		\end{aligned}
 	\end{equation}
 	where $C$ depends on $|\phi^{1/n}|_{C^2(\bar M)}$  
 	and other known data but not on $(\inf_M\phi)^{-1}$.
 \end{theorem}

It only requires to prove quantitative boundary estimate.
 \begin{theorem}
 	\label{thm-21-yuan}
 	Under the assumptions of  Theorem \ref{thm0-n-1-yuan3},  any strictly $(n-1)$-PSH solution $u\in C^3(M)\cap C^2(\bar M)$ to the Dirichlet problem \eqref{MA-n-1} and \eqref{bdy-value-2} satisfies
 	\begin{equation}
 		\label{bdy-sec-estimate-quar1}
 		\begin{aligned}
 			\sup_{\partial M} \Delta u \leq C
 			\left(1+\sup_M |\nabla u|^2\right),
 		\end{aligned}
 	\end{equation}
 	where $C$ depends on $|\phi^{{1}/{n}}|_{C^{1}(\bar M)}$, $|\varphi|_{C^3(\bar M)}$,
 	$|\underline{u}|_{C^2(\bar M)}$, $\partial M$ up to third derivatives and other known data.
 		
 	Furthermore, $C$ is independent of $(\inf_M \phi)^{-1}$ if. $\partial M$ is mean pseudoconcave.
 \end{theorem}

 With Theorems \ref{thm-21-yuan} and \ref{STW-estimate1} at hand, we derive \eqref{sec-estimate-quar1} and then establish gradient estimate via the blow-up argument.
 A somewhat remarkable fact to us is that we can prove the quantitative boundary estimate 
 under a weaker assumption 
 $\tilde{\phi}=\phi^{1/(n-1)}\in C^{1,1}(\bar M)$.
We observe in Proposition \ref{remark2.4} that Sz\'ekelyhidi-Tosatti-Weinkove's original proof for second estimate
 also works under such a weak assumption.
 As a consequence we obtain Theorem \ref{them1-n-1-phi}.


 \section{Preliminaries and lemmas}
 \label{sec2}
 
 \subsection{Background of the equation}

The equation \eqref{MA-n-1} 
can be reformulated in the following form
\begin{equation}
	\label{mainequ-gauduchon}
	\begin{aligned}
		\log P_{n-1}(\lambda(\tilde{\mathfrak{g}}[u]))=\psi  \mbox{ and }
		\lambda(\tilde{\mathfrak{g}}[u])\in\mathcal{P}_{n-1}.
	\end{aligned}
\end{equation}
Here 
\begin{equation}
	\label{P_n-1}
	\begin{aligned}
		P_{n-1}(\lambda)=\mu_1\cdots\mu_n, \quad
		\mu_i=\sum_{j\neq i}\lambda_j. 
	\end{aligned}
\end{equation}
 \begin{equation} \label{yuan3-buchong112} \begin{aligned}
		\mathcal{P}_{n-1}=\{\lambda\in \mathbb{R}^n: \mu_i>0, \mbox{  } \forall 1\leq i\leq n\},
\end{aligned} \end{equation}
and $\psi=\log\phi+n\log(n-1)$,   
\begin{equation}
	\label{yuan3-buchong111}
	\begin{aligned}
		\tilde{\mathfrak{g}}_{i\bar j}=u_{i\bar j}+\check{\chi}_{i\bar j}+W_{i\bar j},
	\end{aligned}
\end{equation}
where $\check{\chi}_{i\bar j}=(\mathrm{tr}_\omega \tilde{\chi})g_{i\bar j}-(n-1)\tilde{\chi}_{i\bar j}$,
$W_{i\bar j}=(\mathrm{tr}_\omega Z)g_{i\bar j}-(n-1)Z_{i\bar j}$, 
\begin{equation}
	\label{Z-tensor1}
	\begin{aligned}
		Z_{i\bar j}=\,& \frac{g^{p\bar q}   \bar T^l_{ql} g_{i\bar j}u_p
			+g^{p\bar q} T^{k}_{pk} g_{i\bar j}u_{\bar q}
			-g^{k\bar l}g_{i\bar q} \bar T^q_{lj}u_k
			-g^{k\bar l}g_{q\bar j}T^{q}_{ki}u_{\bar l}
			-\bar T^{l}_{jl}u_i -T^{k}_{ik}u_{\bar j} }{2(n-1)}, 
	\end{aligned}
\end{equation}
where 
$T^{k}_{ij}= g^{k\bar l}  (\frac{\partial g_{j\bar l}}{\partial z_i} - \frac{\partial g_{i\bar l}}{\partial z_j}),$
see also \cite{STW17}.  

For simplicity we denote
\begin{equation}
	\begin{aligned}
		f(\lambda)=\log P_{n-1}(\lambda),
	\end{aligned}
\end{equation}
For any $\lambda\in \mathcal{P}_{n-1}$ we have the following simple properties:
\begin{equation}
	\label{formular26}
	\begin{aligned}
	  f_i(\lambda)=\sum_{j\neq i} \frac{1}{\mu_j}, \quad
		\sum_{i=1}^n f_i(\lambda)=(n-1)\sum_{i=1}^n \frac{1}{\mu_i}, 
	\end{aligned}
\end{equation}
\begin{equation}
	\label{yuan-37}
	\begin{aligned}
		f_i(\lambda)\geq \frac{\min_k{\mu_k}}{n\max_k{\mu_k}}\sum_{j=1}^n   f_j(\lambda), \quad \forall i,
	\end{aligned}
\end{equation}
\begin{equation}
	\label{formular27}
	\begin{aligned}
		\sum_{i=1}^n f_i(\lambda)\lambda_i=n, 
		\quad
		\sum_{i=1}^n f_i(\lambda)\underline{\lambda}_i=\sum_{i=1}^n\frac{\underline{\mu}_i}{\mu_i}.
	\end{aligned}
\end{equation}
Moreover 
\begin{equation}
	\label{sumf_i1}
	\begin{aligned}
		\sum_{i=1}^n f_i(\lambda)\geq
		 n (n-1)e^{-\sigma/n},
	  \mbox{ for } f(\lambda)=\sigma,
	\end{aligned}
\end{equation}

\begin{lemma}
	\label{type-lemma1}
	Let $\underline{\lambda}\in \mathcal{P}_{n-1} $ and $\varepsilon=\inf_M \min_{i}\underline{\mu}_i$ (then $\varepsilon>0$). 
	Suppose that
	\begin{equation}
		\label{key-yuanrr-1}
		\begin{aligned}
			\min_j\mu_j\leq \frac{\varepsilon}{2n}.
		\end{aligned}
	\end{equation}
	Then
	\begin{equation}
		\label{sumfi-3}
		\begin{aligned}
			\sum_{i=1}^n\frac{1}{\mu_i}\geq \frac{2n}{\varepsilon},
		\end{aligned}
	\end{equation}
	\begin{equation}
		\begin{aligned}
			\sum_{i=1}^n f_i(\lambda)(\underline{\lambda}_i-\lambda_i)\geq \frac{\varepsilon}{2}\sum_{i=1}^n \frac{1}{\mu_i}=
		\frac{\varepsilon}{2(n-1)}\sum_{i=1}^n f_i(\lambda).
		\end{aligned}
	\end{equation}
	
\end{lemma}

\begin{remark}
	\label{lemma2.2}
	It only requires to consider the case when \eqref{key-yuanrr-1} holds. Otherwise for   
	$ \sum_{i=1}^n \log\mu_i=\sigma$, we have
	\begin{equation}
		\begin{aligned}
			\max_i\mu_i \leq (2n)^{n-1}e^\sigma/\varepsilon^{n-1}, \\ 
		\end{aligned}
	\end{equation}
and
\begin{equation}
	\frac{\min_i\mu_i}{\max_i\mu_i}\geq \frac{\varepsilon^n}{(2n)^{n}e^\sigma}.
	\end{equation}
\end{remark}

 \vspace{1mm}
 Throughout this paper we denote 
 \[\mathfrak{\tilde{g}}=\mathfrak{\tilde{g}}[u],  
 \mbox{  } \mathfrak{\underline{\tilde{g}}}=\mathfrak{\tilde{g}}[\underline{u}],  
 \mbox{ } \underline{Z}=Z[\underline{u}],
 \mbox{  } Z=Z[u], 
 \mbox{  } W=W[u],
 \mbox{ } \underline{W}=W[\underline{u}]\]
 for solution $u$ and subsolution $\underline{u}$. In Sections \ref{sec3} and \ref{sec4} we denote
 \begin{equation}
 	\begin{aligned}
 		\,& \lambda=\lambda(\mathfrak{\tilde{g}}), \,& \underline{\lambda}=\lambda(\mathfrak{\underline{\tilde{g}}}).
 	\end{aligned}
 \end{equation}

 \subsection{Some computation and notation}
 Fix $x_0\in \partial M$ we can choose a local holomorphic coordinate system
 \begin{equation}
 	\label{goodcoordinate1}
 	\begin{aligned}
 		(z_1,\cdots, z_n), \mbox{  } z_i=x_i+\sqrt{-1}y_i, 
 	\end{aligned}
 \end{equation}
 centered at $x_0$,
 such that $g_{i\bar j}(0)=\delta_{ij}$, $\frac{\partial}{\partial x_n}$ is the inner normal vector at the origin, and
 $T^{1,0}_{{x_0},{\partial M}}$ is spanned by $\frac{\partial}{\partial z_\alpha}$ for $1\leq\alpha\leq n-1$.
 We   denote $\sigma(z)$  the distance function from $z$ to $\partial M$ with respect to $\omega$. 
 Also we denote 
 \begin{equation}
 	\label{def-omega1}
 	\begin{aligned}
 		\,& \Omega_\delta:=\{z\in M: |z|<\delta\}, \,& M_\delta:=\{z\in M: \sigma(z)<\delta\}.
 	\end{aligned}
 \end{equation}
 
 In the computation we use derivatives with respect to  Chern connection $\nabla$ of $\omega$,
 and write
 $\partial_{i}=\frac{\partial}{\partial z_{i}}$, 
 $\overline{\partial}_{i}=\frac{\partial}{\partial \bar z_{i}}$,
 $\nabla_{i}=\nabla_{\frac{\partial}{\partial z_{i}}}$,
 $\nabla_{\bar i}=\nabla_{\frac{\partial}{\partial \bar z_{i}}}$.
 
 For a smooth function $v$,
 $$
 v_i:=
 \partial_i v,  
 \mbox{  } v_{\bar i}:=
 \partial_{\bar i} v,
 \mbox{  } 
 v_{i\bar j}:= 
 \partial_i\overline{\partial}_j v, 
 \mbox{  } 
 v_{ij}:=
 \nabla_{j}\nabla_{i} v =
 \partial_i \partial_j v -\Gamma^k_{ji}v_k, \cdots \mbox{etc},
 $$
 where $\Gamma_{ij}^k$ are the Christoffel symbols
 defined  by 
 $\nabla_{\frac{\partial}{\partial z_i}} \frac{\partial}{\partial z_j}=\Gamma_{ij}^k \frac{\partial}{\partial z_k}.$
 
 The boundary value condition implies 
 \begin{equation}
 	\label{yuan3-buchong5}
 	\begin{aligned}
 		\,&  u_\alpha(0)=\underline{u}_\alpha(0), \,&
 		u_{\alpha\bar\beta}(0)=\underline{u}_{\alpha\bar\beta}(0)+(u-\underline{u})_{x_n}(0)\sigma_{\alpha\bar\beta}(0)
 	\end{aligned}
 \end{equation}
 for $1\leq \alpha, \beta\leq n-1$.  
 Let $\check{u}$ be the solution to 
 \begin{equation}
 	\label{supersolution-1}
 	\begin{aligned}
 		\,&  \mathrm{tr}_{\omega} (\mathfrak{\tilde{g}}[\check{u}])=0 \mbox{ in } M,   \,& \check{u}=\varphi \mbox{ on } \partial M.
 	\end{aligned}
 \end{equation}
 The existence of $\check{u}$ follows from standard theory of elliptic equations.
 By the maximum principle and boundary value condition,  one derives
 the following:
 \begin{lemma}
 \begin{equation}
 	\label{key-14-yuan3}
 	\begin{aligned}
 		\underline{u}\leq u\leq \check{u} \mbox{ in } M, \quad 0\leq (u-\underline{u})_{x_n}(0) \leq (\check{u}-\underline{u})_{x_n}(0).  
 	\end{aligned}
 \end{equation}
 In particular,
 \begin{equation}
 	\label{c-0-c1b}
 	\begin{aligned}
 		\sup_M|u|+\sup_{\partial M}|\nabla u| \leq C.
 	\end{aligned}
 \end{equation}
\end{lemma}

 By \eqref{Z-tensor1}, \eqref{yuan3-buchong5} and $W[v]_{i\bar j}=(\mathrm{tr}_\omega Z[v])g_{i\bar j}-(n-1)Z[v]_{i\bar j}$, one can verify the following:
 \begin{lemma}
 	\label{key-lemma34}
 	At the origin ($\{z=0\}$),
 \begin{equation}
 	\label{yuan3-buchong12}
 	\begin{aligned} 
 		\sum_{\alpha=1}^{n-1}\tilde{\mathfrak{g}}[v]_{\alpha\bar\alpha} 
 		=  \sum_{\alpha=1}^{n-1}(v_{\alpha\bar\alpha}+ \check{\chi}_{\alpha\bar\alpha})+ \sum_{\alpha=1}^{n-1}W[v]_{\alpha\bar\alpha},
 	\end{aligned}
 \end{equation}
 \begin{equation}
 	\label{stw-44}
 	\begin{aligned}
 		2(n-1)Z[v]_{n\bar n}
 		=\,&  \sum_{\alpha, \beta=1}^{n-1}  (\bar T^\beta_{\alpha \beta}  v_\alpha
 		+ T^{\beta}_{\alpha \beta}  v_{\bar \alpha}), 
 	\end{aligned}
 \end{equation}

 \begin{equation}
 	\label{lemma-B5}
 	\begin{aligned} 
 		\sum_{\alpha=1}^{n-1}\tilde{\mathfrak{g}}_{\alpha\bar\alpha} =
 		\sum_{\alpha=1}^{n-1}\underline{\tilde{\mathfrak{g}}}_{\alpha\bar\alpha}
 		+(u-\underline{u})_{x_n}\sum_{\alpha=1}^{n-1}\sigma_{\alpha\bar\alpha}. 
 	\end{aligned}
 \end{equation}

 If in addition we take
 \begin{equation}
 	\label{w-rewrite1}
 	\begin{aligned}
 		W[v]_{i\bar j} = W_{i\bar j}^k v_k + W_{i\bar j}^{\bar k} v_{\bar k}, 
 	\end{aligned}
 \end{equation}
 then at the origin  
 \begin{equation}
 	\label{lemma-B4}
 	\begin{aligned}
 		\sum_{\alpha=1}^{n-1} W_{\alpha\bar\alpha}^n = \sum_{\alpha=1}^{n-1} W_{\alpha\bar\alpha}^{\bar n}=0.
 	\end{aligned}
 \end{equation}
 \end{lemma}

 \section{Quantitative boundary estimate for pure normal derivative}
 \label{sec3}
 
 The goal of this section is to derive a quantitative version of boundary estimate for pure normal derivative.
 Before stating it, we   denote  an orthonormal basis of 
 $T^{1,0}_{\partial M}:= T^{1,0}_{\bar M} \cap T^{\mathbb{C}}_{ \partial M}$ by
 \begin{equation}
 	\label{xi1-yuan3}
 	\begin{aligned}
 		\xi_1, \cdots, \xi_{n-1}. 
 	\end{aligned}
 \end{equation} 
 As above, $\nu$ denotes the unit inner normal vector along the boundary. Let
 \begin{equation}
 	\label{xi2-yuan3}
 	\begin{aligned}
 		\xi_n=\frac{1}{\sqrt{2}}\left(\mathrm{{\bf \nu}}-\sqrt{-1}J\nu\right). 
 	\end{aligned}
 \end{equation}

 \begin{proposition}
 	\label{proposition-quar-yuan2}
 	Let $(M, J,\omega)$ be a compact Hermitian manifold with $C^3$ boundary.
 	In addition we assume \eqref{subsolution-MA-n-1} is satisfied.
 	For any strictly $(n-1)$-PSH solution  $u\in  C^{2}(\bar M)$ to the 
 	Dirichlet problem \eqref{MA-n-1} and \eqref{bdy-value-2}, 
 	then there is a uniform positive constant $C$ depending on   
 	$|u|_{C^0(M)}$,   $|\nabla u|_{C^0(\partial M)}$,  $\sup_M\phi$,
 	$|\underline{u}|_{C^{2}(\bar M)}$,  $\partial M$ up to third derivatives 
 	and other known data, such that 
 	\begin{equation}
 		\label{yuan-prop2}
 		\begin{aligned}
 			\mathrm{tr}_\omega(\mathfrak{\tilde{g}})(x_0)  \leq C\left(1 +  \sum_{\alpha=1}^{n-1} |\mathfrak{\tilde{g}}(\xi_\alpha, J\bar \xi_n)(x_0)|^2\right),  \mbox{  } \forall x_0\in\partial M.
 		\end{aligned}
 	\end{equation}
 	Moreover, if $\partial M$ is mean pseudoconcave then the constant $C$  is independent of  $(\inf_M\phi)^{-1}$
 	and only depends on $\partial M$ up to second derivatives and other known data under control.
 \end{proposition}

 \subsection{First ingredient and its proof}
 
 Let $\eta=(u-\underline{u})_{x_n}(0)$. 
 We know that $\eta\geq 0$. 
 Let  
 \begin{equation}
 	\label{t0-2}
 	\begin{aligned}
 		t_0=-{\eta \sum_{\alpha=1}^{n-1}\sigma_{\alpha\bar\alpha}(0)}/{\sum_{\alpha=1}^{n-1}\underline{\tilde{\mathfrak{g}}}_{\alpha\bar\alpha}(0)}.
 	\end{aligned}
 \end{equation}
 Obviously, $t_0<1$.   
 From \eqref{lemma-B5}, we have at the origin
 \begin{equation}
 	\label{t0-2-buchong1}
 	\begin{aligned} 
 		\sum_{\alpha=1}^{n-1}\tilde{\mathfrak{g}}_{\alpha\bar\alpha} =
 		(1-t_0) \sum_{\alpha=1}^{n-1}\underline{\tilde{\mathfrak{g}}}_{\alpha\bar\alpha}.
 	\end{aligned}
 \end{equation}

 \begin{lemma}
 	\label{yuan-k2v}
 	There is a uniform positive constant $C$ depending on $(1-t_0)^{-1}$, $\sup_M\phi$ and other known data such that 
 	\begin{equation}
 		\begin{aligned}
 			\mathrm{tr}_\omega(\mathfrak{\tilde{g}})\leq C\left(1+\sum_{\alpha=1}^{n-1}|\mathfrak{\tilde{g}}_{\alpha\bar n}|^2\right).  \nonumber
 		\end{aligned}
 	\end{equation}
 \end{lemma}
 
 \begin{proof}
 	
 	The argument 
 	is based on Lemma \ref{yuan's-quantitative-lemma} proposed in \cite{yuan2017} (also see \cite{yuan-V}).   Around $x_0$ we use the local holomorphic coordinates  
 	that we have chosen in \eqref{goodcoordinate1};
 	furthermore, we assume that $ ({\tilde{\mathfrak{g}}}_{\alpha\bar\beta})$ is diagonal at the origin $x_0=\{z=0\}$.
 	In the proof the discussion is done at the origin, and 
 	the Greek letters, such as $\alpha, \beta$, range from $1$ to $n-1$.
 	Let's denote
 	\begin{equation}
 		{\tilde{A}}(R)=\left(
 		\begin{matrix}
 			R-\tilde{\mathfrak{{g}}}_{1\bar 1}&&  &-\tilde{\mathfrak{g}}_{1 \bar n}\\
 			&\ddots&&\vdots \\
 			& &  R-\tilde{\mathfrak{{g}}}_{{(n-1)} \overline{(n-1)}}&- \tilde{\mathfrak{g}}_{(n-1) \bar n}\\
 			-\tilde{\mathfrak{g}}_{n \bar 1}&\cdots& -\tilde{\mathfrak{g}}_{n \overline{(n-1)}}& \sum_{\alpha=1}^{n-1}\tilde{\mathfrak{g}}_{\alpha \bar\alpha}  \nonumber
 		\end{matrix}
 		\right),
 	\end{equation}
 	\begin{equation}
 		\tilde{\underline{A}}(R)=\left(
 		\begin{matrix}
 			R-\tilde{\mathfrak{{g}}}_{1\bar 1}&&  &-\tilde{\mathfrak{g}}_{1\bar n}\\
 			&\ddots&&\vdots \\
 			& & R-\tilde{\mathfrak{{g}}}_{{(n-1)}  \overline{(n-1)}}&- \tilde{\mathfrak{g}}_{(n-1) \bar n}\\
 			-\tilde{\mathfrak{g}}_{n \bar1}&\cdots& -\tilde{\mathfrak{g}}_{n \overline{(n-1)}}& (1-t_0)\sum_{\alpha=1}^{n-1}\underline{\tilde{\mathfrak{g}}}_{\alpha\bar\alpha} \nonumber
 		\end{matrix}
 		\right).
 	\end{equation}
 	In particular, when $R=\mathrm{tr}_\omega (\tilde{\mathfrak{{g}}})$, ${\tilde{A}}(R)=\mathrm{tr}_\omega(\mathfrak{\tilde{g}}) \omega-\mathfrak{\tilde{g}}$.
 	By \eqref{lemma-B5} and \eqref{t0-2},
 	$$\tilde{\underline{A}}(R)={\tilde{A}}(R).$$
 	
 	One can see that there is a uniform positive constant $R_0$ depending on $(1-t_0)^{-1}$ and $(\inf_{\partial M}\mathrm{dist}(\underline{\lambda},\partial \Gamma_n))^{-1}$ (but not on $(\inf_M\phi)^{-1}$) such that
 	$${f}\left(R_0,\cdots, R_0,(1-t_0)\sum_{\alpha=1}^{n-1}\underline{\tilde{\mathfrak{g}}}_{\alpha \bar\alpha}\right) > \psi.$$
 	Therefore, there is a positive constant $\varepsilon_{0}$, 
 	depending  on $\inf_{\partial M}\mathrm{dist}(\underline{\lambda},\partial \Gamma_n)$,  such that 
 	\begin{equation}
 		\label{opppp-Gauduchon}
 		\begin{cases}
 			f\left(R_0-\varepsilon_{0},\cdots, R_0-\varepsilon_{0},(1-t_0)\sum_{\alpha=1}^{n-1}\underline{\tilde{\mathfrak{g}}}_{\alpha\bar\alpha}-\varepsilon_{0}\right)\geq \psi, 
 			\\
 			  R_0>\varepsilon_{0}, \mbox{ }  (1-t_0)\sum_{\alpha=1}^{n-1}\underline{\tilde{\mathfrak{g}}}_{\alpha\bar\alpha}>\varepsilon_{0}.
 		\end{cases}
 	\end{equation}

 	Note that 
 	\begin{equation}
 		{\tilde{A}}(R)={\tilde{\underline{A}}}(R)=RI_n-\left(
 		\begin{matrix}
 			\tilde{\mathfrak{{g}}}_{1\bar 1}&&  &\tilde{\mathfrak{g}}_{1 \bar n}\\
 			&\ddots&&\vdots \\
 			& &  \tilde{\mathfrak{{g}}}_{{(n-1)} \overline{(n-1)}}& \tilde{\mathfrak{g}}_{(n-1) \bar n}\\
 			\tilde{\mathfrak{g}}_{n \bar 1}&\cdots& \tilde{\mathfrak{g}}_{n \overline{(n-1)}}&
 			R-(1-t_0) \sum_{\alpha=1}^{n-1}\tilde{\mathfrak{\underline{g}}}_{\alpha \bar\alpha}  \nonumber
 		\end{matrix}
 		\right)
 	\end{equation}
 	here $I_n= \left(\delta_{ij}\right)$.
 	Let's pick  $\epsilon=\frac{\varepsilon_0(1-t_0)}{2(n-1)}$ in Lemma \ref{yuan's-quantitative-lemma} then set
 	\begin{equation}
 		\begin{aligned}
 			R_s=\,& \frac{2(n-1)(2n-3)}{\varepsilon_0(1-t_0)}
 			\sum_{\alpha=1}^{n-1} | \tilde{\mathfrak{g}}_{\alpha \bar n}|^2
 			+ (n-1)\sum_{\alpha=1}^{n-1} | \tilde{\mathfrak{{g}}}_{\alpha \bar\alpha}| 
 			+ (1-t_0)\sum_{\alpha=1}^{n-1} | \tilde{\mathfrak{\underline{g}}}_{\alpha \bar\alpha}|
 			+R_0, \nonumber
 		\end{aligned}
 	\end{equation}
 	where $\varepsilon_0$ and $R_0$ are constants from \eqref{opppp-Gauduchon}.
 	Let $\lambda(\tilde{\underline{A}}(R_s))=(\lambda_1(R_s),\cdots,\lambda_n(R_s))$ denote
 	the eigenvalues of $\tilde{\underline{A}}(R_s)$.
 	According to Lemma  \ref{yuan's-quantitative-lemma},
 	\begin{equation}
 		\label{lemma12-yuan-Gauduchon}
 		\begin{aligned}
 			\lambda_\alpha(R_s) \geq \,& R_s- \tilde{\mathfrak{g}}_{\alpha\bar \alpha}-\frac{\varepsilon_0}{2(n-1)},
 			\quad \forall 1\leq \alpha \leq n-1, \\
 			\lambda_n(R_s) \geq \,& (1-t_0)\sum_{\alpha=1}^{n-1}\tilde{\mathfrak{\underline{g}}}_{\alpha\bar \alpha}
 			-\frac{\varepsilon_0}{2}.  \nonumber
 		\end{aligned}
 	\end{equation}
 	Therefore 
 	\begin{equation}
 		\label{puretangential2-gauduchon}
 		\begin{aligned}
 			{f}(\lambda(\tilde{{A}}(R_s)))\geq \psi. \nonumber
 		\end{aligned}
 	\end{equation}
 	We get
 	$$\mathrm{tr}_\omega(\tilde{\mathfrak{g}}) \leq R_s.$$
 	
 	
 \end{proof}

 \subsection{Second ingredient and the proof}
 
 According to Lemma \ref{yuan-k2v} 
 it requires only to prove that $(1-t_0)^{-1}$ can be uniformly bounded from above. That is,
 \begin{equation}
 	\label{1-t_0}
 	\begin{aligned}
 		(1-t_0)^{-1}\leq C.
 	\end{aligned}
 \end{equation}
 
 \subsubsection*{\bf Case 1: The boundary $\partial M$ is mean pseudoconcave.}
 Note $\eta=(u-\underline{u})_{x_n}(0)\geq 0$.
 The   \textit{mean pseudoconcavity} of boundary gives   $$t_0\leq0$$
 which automatically implies  \eqref{1-t_0}.

 \subsubsection*{\bf Case 2: Without mean pseudoconcave restriction to  $\partial M$.}
 \begin{lemma}
 	\label{key-lemma-B1}
 	The inequality \eqref{1-t_0} holds for a uniform positive constant $C$ depending on $(\inf_M\phi)^{-1}$, $|u|_{C^0(\bar M)}$,  $|\nabla u|_{C^0(\partial M)}$, $|\underline{u}|_{C^2(\bar M)}$, 
 	$\partial M$ up to third derivatives and other known data.
 	
 \end{lemma}
 
 We assume throughout $\eta>0$; otherwise $t_0=0$ and we have done.
 From $\lambda(\underline{\tilde{\mathfrak{g}}})\in \mathcal{P}_{n-1}$ we know $\sum_{\alpha=1}^{n-1}\underline{\tilde{\mathfrak{g}}}_{\alpha\bar\alpha}(0)>0$.
 In what follows we assume $$t_0>\frac{1}{2}.$$ Since $\eta$ has a uniform upper bound, thus at origin
 \begin{equation}
 	\label{yuan3-a2}
 	-\sum_{\alpha=1}^{n-1}\sigma_{\alpha\bar\alpha}(0)\geq {t_0\sum_{\alpha=1}^{n-1}\underline{\tilde{\mathfrak{g}}}_{\alpha\bar\alpha}(0)}/{\eta}\geq  {\sum_{\alpha=1}^{n-1}\underline{\tilde{\mathfrak{g}}}_{\alpha\bar\alpha}(0)}/{2\eta}\geq a_2
 \end{equation}
 where $a_2= {\inf_{z\in\partial M}\sum_{\alpha=1}^{n-1}\underline{\tilde{\mathfrak{g}}}_{\alpha\bar\alpha}(z)}/{2\sup_{\partial M}|\nabla (u-\underline{u})|}.$

Let $\Omega_\delta$ be as in \eqref{def-omega1}.
 Inspired by an idea of Caffarelli-Nirenberg-Spruck \cite{CNS3} (see also \cite{LiSY2004}),
  we set on $\Omega_\delta$
 \begin{equation}
 	\begin{aligned}
 		d(z)=\sigma(z)+\tau |z|^2 \nonumber
 	\end{aligned}
 \end{equation}
 where $\tau$ is a positive constant 
 to be determined. Let 
 \begin{equation}
 	\label{w-buchong1}
 	\begin{aligned}
 		w(z)=\underline{u}(z)+({\eta}/{t_0})\sigma(z)+l(z)\sigma(z)+Ad(z)^2. 
 	\end{aligned}
 \end{equation}
 where $A$ is a positive constant to be determined, and $l(z)=\sum_{i=1}^n(l_iz_i+\bar l_{ i} \bar z_{i})$ 
 where $l_i\in \mathbb{C}$, $\bar l_i=l_{\bar i}$ to be chosen as in \eqref{chosen-2}.

 Let $T_1(z),\cdots, T_{n-1}(z)$ be an orthonormal basis for holomorphic tangent space  of level hypersurface
 $\{w: d(w)=d(z)\}$ at $z$, so that at the origin 
 $T_\alpha(0)= \frac{\partial }{\partial z_\alpha}$  for each $1\leq\alpha\leq n-1$. Furthermore, let $T_n=\frac{\partial d}{|\partial d|}$.

 Such a basis exists: 
 We see at the origin  $\partial d(0)=\partial \sigma(0)$.  
 Thus for $1\leq \alpha\leq n-1$, we can choose $T_\alpha$ such that at the  origin
 $T_\alpha(0)= \frac{\partial }{\partial z_\alpha}$.
 
 For a real $(1,1)$-form $\Theta=\sqrt{-1}\Theta_{i\bar j}dz_i\wedge d\bar z_j$, we denote by 
 $\lambda(\Theta)$  the eigenvalues of $\Theta$ (with respect to $\omega$) with order $\lambda_1(\Theta)\leq \cdots\leq \lambda_n(\Theta)$. Since $\lambda_n(\Theta)\geq \Theta(T_n,J\bar T_n)$, one has
 \begin{equation}
 	\label{lemma-yuan3-buchong1}
 	\begin{aligned}
 		\sum_{\alpha=1}^{n-1}\lambda_\alpha(\Theta)\leq \sum_{\alpha=1}^{n-1} \Theta(T_\alpha,J\bar T_\alpha).
 	\end{aligned}
 \end{equation}
 We then define $\Lambda(\Theta):=\sum_{\alpha=1}^{n-1} T_{\alpha}^i \bar T_{\alpha}^j \Theta_{i\bar j}$ for
 $\Theta=\sqrt{-1}\Theta_{i\bar j}dz_i\wedge d\bar z_j.$
 
 \begin{lemma}
 	\label{key-lemma-B2}
 	There are 
 	$\tau$, $\delta$, $A$ and 
 	$l(z)$ depending on 
 	$|u|_{C^0(M)}$, 
 	$|\nabla u|_{C^0(\partial M)}$,  
 	$|\underline{u}|_{C^2(M)}$, 
 	$\partial M$ up to third derivatives and other known data so that
 	\begin{equation}
 		\begin{aligned}
 			\,&   \Lambda(\tilde{\mathfrak{g}}[w]) \leq0 \mbox{ in } \Omega_\delta, \,& u\leq w \mbox{ on } \partial \Omega_\delta. \nonumber
 		\end{aligned}
 	\end{equation}
 	
 \end{lemma}

 \begin{proof}
In the proof, Lemma \ref{key-lemma34} plays key roles in treating the gradient terms from $Z$ in equation.
 	Direct computations give
 	\begin{equation} \begin{aligned}
 			\,& w_i = \underline{u}_i + \frac{\eta}{t_0} \sigma_i +l_i\sigma +l(z)\sigma_i +2Add_i, \\ \nonumber
 			w_{i\bar j}=
 			\underline{u}_{i\bar j}\,&+\frac{\eta}{t_0} \sigma_{i\bar j}+l(z)\sigma_{i\bar j}
 			+(l_i\sigma_{\bar j}+\sigma_i l_{\bar j})
 			+2Add_{i\bar j} +2A d_i d_{\bar j}.  \nonumber
 		\end{aligned}
 	\end{equation}
 	Then
 	\begin{equation}
 		\begin{aligned}
 			\Lambda(\tilde{\mathfrak{g}}[w])=\,& 
 			\sum_{\alpha=1}^{n-1}T_\alpha^{i}\bar T_\alpha^j ( (\check{\chi}_{i\bar j}+ \underline{w}_{i\bar j}+W_{i\bar j}^pw_p + W_{i\bar j}^{\bar q}w_{\bar q}) \\
 			=\,&   \sum_{\alpha=1}^{n-1}T_\alpha^{i}\bar T_\alpha^j  
 			(\check{\chi}_{i\bar j}+ \underline{u}_{i\bar j}+W_{i\bar j}^p \underline{u}_p+W_{i\bar j}^{\bar q}\underline{u}_{\bar q}+\frac{\eta}{t_0} \sigma_{i\bar j})
 			\\\,&
 			+l(z) \sum_{\alpha=1}^{n-1}T_\alpha^{i}\bar T_\alpha^j \sigma_{i\bar j}
 			+\sum_{\alpha=1}^{n-1}T_\alpha^{i}\bar T_\alpha^j (\sigma_i l_{\bar j}+l_i\sigma_{\bar j})
 			\\ \,&
 			+2Ad(z)\sum_{\alpha=1}^{n-1}T_\alpha^{i}\bar T_\alpha^j d_{i\bar j}
 			+\frac{\eta}{t_0} \sum_{\alpha=1}^{n-1}T_\alpha^{i}\bar T_\alpha^j (W_{i\bar j}^p \sigma_p+ W_{i\bar j}^{\bar q} \sigma_{\bar q})\\\,&
 			+l(z)\sum_{\alpha=1}^{n-1}T_\alpha^{i}\bar T_\alpha^j (W_{i\bar j}^p\sigma_p+W_{i\bar j}^{\bar q}\sigma_{\bar q})
 			+\sum_{\alpha=1}^{n-1}T_\alpha^{i}\bar T_\alpha^j (W_{i\bar j}^pl_p+W_{i\bar j}^{\bar q}l_{\bar q})\sigma
 			\\\,&
 			+2Ad(z)\sum_{\alpha=1}^{n-1}T_\alpha^{i}\bar T_\alpha^j(W_{i\bar j}^pd_p+W_{i\bar j}^{\bar q}d_{\bar q}).  \nonumber
 		\end{aligned}
 	\end{equation}

 	\begin{itemize}
 		\item   At origin $z=0$, $T_{\alpha}^i=\delta_{\alpha i}$, 
 		\begin{equation}
 			\begin{aligned}
 				\,& \sum_{\alpha=1}^{n-1}T_\alpha^{i}\bar T_\alpha^j  
 				(\check{\chi}_{i\bar j}+ \underline{u}_{i\bar j}+W_{i\bar j}^p \underline{u}_p+W_{i\bar j}^{\bar q}\underline{u}_{\bar q}+\frac{\eta}{t_0} \sigma_{i\bar j})(0)
 				\\ =\,&\sum_{\alpha=1}^{n-1} \underline{\tilde{\mathfrak{g}}}_{\alpha\bar\alpha}(0)+\frac{\eta}{t_0}\sum_{\alpha=1}^{n-1}\sigma_{\alpha\bar \alpha}(0)=0.  \nonumber
 			\end{aligned}
 		\end{equation}
 		Thus there are complex constants $k_i$ such that on $\Omega_\sigma$,
 		\begin{equation}
 			\begin{aligned}
 			  \sum_{\alpha=1}^{n-1}T_\alpha^{i}\bar T_\alpha^j  
 				(\check{\chi}_{i\bar j}+ \underline{u}_{i\bar j}+W_{i\bar j}^p \underline{u}_p+W_{i\bar j}^{\bar q}\underline{u}_{\bar q}+\frac{\eta}{t_0} \sigma_{i\bar j}) 
 				= \sum_{i=1}^n(k_iz_i+\bar k_i\bar z_i)+O(|z|^2).  \nonumber
 			\end{aligned}
 		\end{equation}
 		\item  Next, we see
 		\begin{equation}
 			\begin{aligned}
 				2A d(z)  \sum_{\alpha=1}^{n-1}T_\alpha^i\bar T_\alpha^{j} d_{i\bar j} \leq -\frac{Aa_2d(z)}{2},  \nonumber
 			\end{aligned}
 		\end{equation}
 		provided $0<\delta, \tau\ll1$.  Here we use 
 		\begin{equation}
 			\begin{aligned}
 				\sum_{\alpha=1}^{n-1}T_\alpha^i\bar T_\alpha^{j} d_{i\bar j}
 				=\,& (\sum_{\alpha=1}^{n-1}T_\alpha^i\bar T_\alpha^{j}-\sum_{\alpha=1}^{n-1}T_\alpha^i\bar T_\alpha^{j} (0) )d_{i\bar j}
 				+\sum_{\alpha=1}^{n-1}\sigma_{\alpha\bar\alpha}(z)+(n-1)\tau \\
 				\leq \,& (n-1)\tau-a_2+O(|z|) \leq -\frac{a_2}{4} \nonumber
 			\end{aligned}
 		\end{equation}
 		by  
 		$\sum_{\alpha=1}^{n-1}\sigma_{\alpha\bar\alpha}(z)=\sum_{\alpha=1}^{n-1}\sigma_{\alpha\bar\alpha}(0)+O(|z|)$,  \eqref{yuan3-a2} and
 		\begin{equation}
 			\label{yuan3-buchong2}
 			\begin{aligned}
 				\sum_{\alpha=1}^{n-1}   T_{\alpha}^i \bar T_{\alpha}^j(z)
 				=\sum_{\alpha=1}^{n-1}   T_{\alpha}^i \bar T_{\alpha}^j(0)+O(|z|).
 			\end{aligned}
 		\end{equation}
 		
 		\item   
 		\begin{equation}
 			\begin{aligned}
 				\,& l(z) \sum_{\alpha=1}^{n-1}T_\alpha^{i}\bar T_\alpha^j \sigma_{i\bar j}
 				+\sum_{\alpha=1}^{n-1}T_\alpha^{i}\bar T_\alpha^j (\sigma_i l_{\bar j}+l_i\sigma_{\bar j}) \\
 				= \,&
 				l(z)\sum_{\alpha=1}^{n-1}\sigma_{\alpha\bar\alpha}(0)
 				- \tau\sum_{\alpha=1}^{n-1}  (z_\alpha l_\alpha+\bar z_{\alpha} \bar l_{\alpha})
 				+O(|z|^2)  \nonumber
 			\end{aligned}
 		\end{equation}
 		since by \eqref{yuan3-buchong2} and 
 		$ \sum_{i=1}^n T_\alpha^i \sigma_i=-\tau\sum_{i=1}^n T_\alpha^i
 		\bar z_i $
 		one has
 		\begin{equation}
 			\begin{aligned}
 				l(z) \sum_{\alpha=1}^{n-1}  T^i_\alpha \bar T^j_\alpha \sigma_{i\bar j} 
 				=
 				l(z) \sum_{\alpha=1}^{n-1}\mu_\alpha\sigma_{\alpha\bar\alpha}(0)
 				+O(|z|^2)  \nonumber
 			\end{aligned}
 		\end{equation}
 		\[ \sum_{\alpha=1}^{n-1}T_\alpha^{i}\bar T_\alpha^j (\sigma_i l_{\bar j}+l_i\sigma_{\bar j})=-\tau\sum_{\alpha=1}^{n-1}  (z_\alpha l_\alpha+\bar z_{\alpha}  \bar l_{\alpha})
 		+O(|z|^2).\]

 		\item 
 		At the origin, 
 		\begin{equation}
 			\begin{aligned}
 				\,&  \sum_{\alpha=1}^{n-1}T_\alpha^{i}\bar T_\alpha^j (W_{i\bar j}^p \sigma_p+ W_{i\bar j}^{\bar q} \sigma_{\bar q}) (0)\\
 				=\,&  \sum_{\alpha,\beta=1}^{n-1}  (W_{\alpha\bar \alpha}^\beta \sigma_\beta+ W_{\alpha\bar \alpha}^{\bar \beta} \sigma_{\bar \beta}) (0)  + \sum_{\alpha=1}^{n-1} (W_{\alpha\bar \alpha}^n \sigma_n+ W_{\alpha\bar \alpha}^{\bar n} \sigma_{\bar n})(0)=0, \nonumber
 			\end{aligned}
 		\end{equation}
 		since $\sigma_\beta(0)=0$, and by  \eqref{lemma-B4}
 		\[\sum_{\alpha=1}^{n-1}W_{\alpha\bar\alpha}^n(0)=0, \mbox{  }\sum_{\alpha=1}^{n-1}W_{\alpha\bar\alpha}^{\bar n}(0)=0.\]
 		Thus on $\Omega_\sigma$, 
 		\begin{equation}
 			\begin{aligned}
 				l(z)\sum_{\alpha=1}^{n-1}T_\alpha^{i}\bar T_\alpha^j (W_{i\bar j}^p\sigma_p+W_{i\bar j}^{\bar q}\sigma_{\bar q})(z)
 				=O(|z|^2), \nonumber
 			\end{aligned}
 		\end{equation}
 		and there are complex constants $m_i$ such that 
 		\begin{equation}
 			\begin{aligned}
 				\frac{\eta}{t_0} \sum_{\alpha=1}^{n-1}T_\alpha^{i}\bar T_\alpha^j
 				(W_{i\bar j}^p \sigma_p+ W_{i\bar j}^{\bar q} \sigma_{\bar q})(z)
 				=\sum_{i=1}^n(m_iz_i+\bar m_i \bar z_i)+O(|z|^2).  \nonumber
 			\end{aligned}
 		\end{equation}
 		\item Similarly,
 		$\sum_{\alpha=1}^{n-1}T_\alpha^{i}\bar T_\alpha^j(W_{i\bar j}^pd_p+W_{i\bar j}^{\bar q}d_{\bar q})(0)=0$, thus on $\Omega_\delta$,
 		$$\sum_{\alpha=1}^{n-1}T_\alpha^{i}\bar T_\alpha^j(W_{i\bar j}^pd_p+W_{i\bar j}^{\bar q}d_{\bar q})(z)=O(|z|)$$ 
 		so
 		\begin{equation}
 			\begin{aligned}
 				2Ad(z)\sum_{\alpha=1}^{n-1}T_\alpha^{i}\bar T_\alpha^j(W_{i\bar j}^pd_p+W_{i\bar j}^{\bar q}d_{\bar q})(z)=Ad(z)O(|z|). \nonumber
 			\end{aligned}
 		\end{equation}

 		\item Finally
 		\begin{equation}
 			\begin{aligned}
 				\sum_{\alpha=1}^{n-1}T_\alpha^{i}\bar T_\alpha^j (W_{i\bar j}^pl_p+W_{i\bar j}^{\bar q}l_{\bar q})\sigma(z) \leq C_1\sigma(z). \nonumber
 			\end{aligned}
 		\end{equation}
 		
 	\end{itemize}
 	Therefore, we get
 	\begin{equation}
 		\begin{aligned}
 			\Lambda(\tilde{\mathfrak{g}}[w])\leq \,& 
 			2\mathfrak{Re}\sum_{\alpha=1}^{n-1}\left[z_\alpha\left(k_\alpha+m_\alpha+ l_\alpha\left(\sum_{\beta=1}^{n-1}\sigma_{\beta\bar\beta}(0)-\tau\right)\right)\right] \\
 			\,&
 			+2\mathfrak{Re} \left[z_n\left(k_n+m_n + l_n\sum_{\beta=1}^{n-1}\sigma_{\beta\bar\beta}(0))\right)\right] \\
 			\,&
 			-\frac{a_2 A d(z)}{2} +Ad(z)O(|z|)+C_1\sigma(z) + O(|z|^2).  \nonumber
 		\end{aligned}
 	\end{equation}
 	We complete the proof if $0<\tau, \delta\ll1$, $A\gg1$, and we set 
 	\begin{equation}
 		\label{chosen-2}
 		\begin{aligned}
 			l_\alpha=-\frac{k_\alpha+m_\alpha}{\sum_{\beta=1}^{n-1}\sigma_{\beta\bar\beta}(0)-\tau} \mbox{ for } 1\leq \alpha\leq n-1, \mbox{ }   l_n=-\frac{k_n+m_n}{\sum_{\beta=1}^{n-1}\sigma_{\beta\bar\beta}(0)}.
 		\end{aligned}
 	\end{equation}
 	We can see each $|l_i|$ is uniformly bounded, since $\sum_{\beta=1}^{n-1}\sigma_{\beta\bar\beta}(0)\leq -a_2<0$.

 	Furthermore, on $\partial M\cap\bar\Omega_\delta$, $u(z)-w(z)=-A\tau^2|z|^4$.
 	On $M\cap\partial \Omega_{\delta}$,
 	\begin{equation}
 		\begin{aligned}
 			u(z)-w(z) 
 			\leq \,& |u-\underline{u}|_{C^0(\Omega_\delta)} 
 			-(2A\tau \delta^2+\frac{\eta}{t_0}-2n \sup_{i}|l_i| \delta)\sigma(z)-A\tau^2\delta^4 \\
 			\leq \,& -\frac{A\tau^2 \delta^4}{2} \nonumber
 		\end{aligned}
 	\end{equation}
 	provided $A\gg1$. 
 \end{proof}
 
 \subsection{Completion of proof of Lemma \ref{key-lemma-B1}}
 
 Let $w$ be the function as in Lemma \ref{key-lemma-B2}.
 From the construction above, we know that there is a uniform positive constant $C_0$ such that
 \[\sup_M |\mathfrak{\tilde{g}}[w]|\leq C_0.\]
 Let $\lambda(\tilde{\mathfrak{g}}[w])=(\lambda_1[w],\cdots,\lambda_n[w])$, let  $\mu_i[w]=\sum_{j\neq i}\lambda_j[w]$ and we assume $\lambda_1[w]\leq \cdots\leq \lambda_n[w]$. 
 Denote by
 \begin{equation}
 	\begin{aligned}
 		\overline{\mathcal{P}}_{n-1}^{\inf_M\psi}=\left\{\lambda\in\mathcal{P}_{n-1}: \sum_{i=1}^n \log\mu_i\geq \inf_M \psi \right\}. \nonumber
 	\end{aligned}
 \end{equation}
 
  Near the origin $x_0=\{z=0\}$, there are complex valued constants $b_{ij}$, $a_{i\bar j}$ with $\bar a_{i\bar j}=a_{j\bar i}$ such that 
 \begin{equation}
 	\label{asym-sigma1}
 	\begin{aligned}
 		\sigma(z)=x_n+\sum_{ i,j=1}^{n} a_{i\bar j} z_i\bar z_j+\mathfrak{Re}\sum_{i,j=1}^nb_{ij}z_iz_j+O(|z|^3).
 	\end{aligned}
 \end{equation}
One can choose a  positive constant $C'$ such that $ x_n\leq C'|z|^2$ on 
 $\partial M\cap \bar{\Omega}_\delta$, 
 there is a positive constant $C_2$ depending only on $M$ and $\delta$ so that  
 $$x_n\leq C_2 |z|^2 \mbox{ } \mbox{ on }\partial\Omega_\sigma.$$
 Let $C_2$ be as above we set ${h}(z)=w(z)+\epsilon (|z|^2-\frac{x_n}{C_2})$.    Thus
 \begin{equation}
 	\begin{aligned}
 		u\leq {h} \mbox{ }  \mbox{ on } \partial \Omega_\delta. \nonumber
 	\end{aligned}
 \end{equation}
 Lemma \ref{key-lemma-B2} and \eqref{lemma-yuan3-buchong1} give
 \begin{equation}
 	\begin{aligned}
 		\sum_{\alpha=1}^{n-1}\lambda_\alpha[w]\leq 0 
 		\mbox{ } \mbox{ in } \Omega_\delta.  \nonumber
 	\end{aligned}
 \end{equation}
 That is, in $\Omega_\delta$,
 \begin{equation}
 	\begin{aligned}
 		\lambda[w]\notin \mathcal{P}_{n-1}, \mbox{ i.e. } \mu[w]\notin\Gamma_n.  \nonumber
 	\end{aligned}
 \end{equation}
 In other words, $\lambda[w]\in X$, where
 \[X=\{\lambda\in\mathbb{R}^n: \lambda\notin \mathcal{P}_{n-1}\}\cap \{\lambda\in\mathbb{R}^n:   |\lambda|\leq C_0\}.\]
 Notice $\mathcal{P}_{n-1}$ is open so $X$ is a compact subset; moreover $X\cap \overline{\mathcal{P}}_{n-1}^{\inf_M\psi}=\emptyset$. 
 So we can deduce that the distance between $\overline{\mathcal{P}}_{n-1}^{\inf_M\psi}$ and $X$ 
 is greater than some positive constant depending on  
 $\inf_M\phi$ and other known data.
 Therefore, there exists an $0<\epsilon\ll1$ depending on $\inf_M\phi$, $\lambda[w]$, torsion tensor and other known data such that
 \begin{equation}
 	\begin{aligned}
 		\lambda[{h}] \notin \overline{\mathcal{P}}_{n-1}^{\inf_M\psi}. \nonumber
 	\end{aligned}
 \end{equation}
 By \cite[Lemma B]{CNS3}, we have $$u\leq {h} \mbox{ } \mbox{ in }\Omega_\delta.$$
 Notice 
 $u(0)=\varphi(0)$ and ${h}(0)=\varphi(0)$, we have  $(u-{h})_{x_n}(0)\leq 0$ then   
 \begin{equation}
 	\begin{aligned}
 		(1-t_0)^{-1}\leq 1+\frac{\eta C_2}{\epsilon}.  \nonumber
 	\end{aligned}
 \end{equation}

 \section{Quantitative boundary estimate for tangential-normal derivatives}
 \label{sec4}

 The remaining goal is to derive quantitative boundary estimate for tangential-normal derivatives.

 \begin{proposition}
 	\label{mix-general}
 	
 	Let $(M,J,\omega)$ be a compact Hermitian manifold with $C^3$-smooth boundary.
 	In addition we assume   \eqref{subsolution-MA-n-1} holds.
 	Then 
 	for any strictly $(n-1)$-PSH solution $u\in C^3(M)\cap C^2(\bar M)$ to the Dirichlet problem \eqref{MA-n-1} and \eqref{bdy-value-2},
 	there is a uniform positive constant $C$ depending on $|\varphi|_{C^{3}(\bar M)}$,   
 	$|\phi^{{1}/{n}}|_{C^1(\bar M)}$, $|\underline{u}|_{C^{2}(\bar M)}$, $|\nabla u|_{C^0(\partial M)}$, $\partial M$ 
 	up to third derivatives
 	and other known data (but neither on $(\inf_M\phi)^{-1}$ nor  on $\sup_{M}|\nabla u|$)
 	such that
 	\begin{equation}
 		\label{quanti-mix-derivative-00} 
 		\begin{aligned}
 			|\mathfrak{\tilde{g}}(\xi_\alpha,J\bar\xi_n)(x_0)| \leq C(1+\sup_M|\nabla u|), \mbox{ } \mbox{  } \forall 1\leq\alpha\leq n-1, \mbox{  } \forall x_0\in\partial M. 
 		\end{aligned}
 	\end{equation}
 \end{proposition}
 

 \subsection{Tangential operators on the boundary}
 \label{Tangential-opera-1}
 For  a  given point $x_0\in \partial M$,
 we choose local holomorphic coordinates \eqref{goodcoordinate1}
 centered at $x_0$ in a neighborhood which we assume to be contained in $M_{\delta}$,
 such that $x_0=\{z=0\}$, $g_{i\bar j}(0)=\delta_{ij}$ and $\frac{\partial}{\partial x_{n}}$ is the interior normal direction to $\partial M$ at $x_0$.
 For convenience we set
 \begin{equation}
 	t_{2k-1}=x_{k}, \ t_{2k}=y_{k},\ 1\leq k\leq n-1;\ t_{2n-1}=y_{n},\ t_{2n}=x_{n}.  \nonumber
 \end{equation}

 We define the tangential operator on $\partial M$ 
 \begin{equation}
 	\label{tangential-oper-general1}
 	\begin{aligned} 
 		\mathcal{T}=\nabla_{\frac{\partial}{\partial t_{\alpha}}}- \widetilde{\eta}\nabla_{\frac{\partial}{\partial x_{n}}}  
 	\quad	\mbox{ for each fixed }
 		1\leq \alpha< 2n,
 	\end{aligned}
 \end{equation}
 where $\widetilde{\eta}=\frac{\sigma_{t_{\alpha}}}{\sigma_{x_{n}}}$, $\sigma_{x_{n}}(0)=1,$ $\sigma_{t_\alpha}(0)=0$.
 One has $\mathcal{T}(u-\varphi)=0$ on $\partial M\cap \bar\Omega_\delta$.
 The boundary value condition also gives for each $1\leq \alpha, \beta<n$, 
 \begin{equation}\label{left-up}
 	(u-\varphi)_{t_{i}t_{j}}(0)= (u-\varphi)_{x_{n}}(0)\sigma_{t_{i}t_{j}}(0)
 	\mbox{  }\forall 1\leq i, j<2n.
 \end{equation}

 Let's turn our attention to the setting of complex manifolds with \textit{holomorphically flat} boundary.
 Given $x_0\in \partial M$, one can pick local holomorphic coordinates  
 \begin{equation}
 	\begin{aligned}
 		\label{holomorphic-coordinate-flat}
 		(z_1,\cdots, z_n), \mbox{  } z_i=x_i+\sqrt{-1}y_i, 
 	\end{aligned}
 \end{equation}
 centered at $x_0$ such that
 $\partial M$ is locally of the form $\mathfrak{Re}(z_n)=0$ and $g_{i\bar j}(x_0)=\delta_{ij}$. 
 Under the holomorphic coordinate \eqref{holomorphic-coordinate-flat}, we can take
 \begin{equation}
 	\label{tangential-oper-Leviflat1}
 	\begin{aligned}
 		\mathcal{T}=D:= \frac{\partial}{\partial x_\alpha}, 
 		\mbox{ }\frac{\partial}{\partial y_\alpha},  
 		\quad 1\leq \alpha\leq n-1.
 	\end{aligned}
 \end{equation}
 It  is noteworthy that such local holomorphic coordinate system \eqref{holomorphic-coordinate-flat} is only needed in the proof  of Proposition \ref{mix-Leviflat}. In addition,  when $M=X\times S$,
 $D= \frac{\partial}{\partial x_{\alpha}},
   \frac{\partial}{\partial y_{\alpha}},$ where  $z'=(z_1,\cdots z_{n-1})$ is a local holomorphic coordinate of $X$.

 For simplicity we denote  the tangential  operator on $\partial M$ by
 \begin{equation}
 	\label{tangential-operator123}
 	\begin{aligned}
 		\mathcal{T}=\nabla_{\frac{\partial}{\partial t_{\alpha}}}-\gamma\widetilde{\eta}\nabla_{\frac{\partial}{\partial x_{n}}}.
 	\end{aligned}
 \end{equation}
 Here  $\gamma=0$ (i.e. $\mathcal{T}=\nabla_{\frac{\partial}{\partial t_{\alpha}}}=D$) if $\partial M$ is holomorphically flat, while  for general boundary we take $\gamma=1$.
 
 From \eqref{asym-sigma1}  we have  $|\widetilde{\eta}|\leq C'|z|$ on $\Omega_\delta$.
 The boundary value condition $(u-\varphi)|_{\partial M}=0$  
 gives $\mathcal{T}(u-\varphi)|_{\partial M}=0.$ Combining with  \eqref{c-0-c1b}, we have
 \begin{equation}\begin{aligned}\label{bdr-t}
 		\mathcal{T}(u-\varphi)=0 \mbox{ and } |(u-\varphi)_{t_{\alpha}}|\leq C|z| \mbox{ } \mbox{ on } \partial M\cap\bar \Omega_\delta,
 		\mbox{  } \forall 1\leq \alpha<2n.
 	\end{aligned}
 \end{equation}

 \subsection{Completion of proof of Proposition \ref{mix-general}}
 \label{Quantitative-boundes-mix}
 
 Let $F(A)=f(\lambda(A))$, and we denote 
 \begin{equation}
 	\begin{aligned}
 		\,& F^{i\bar j}(A)=\frac{\partial F}{\partial a_{i\bar j}}, \,& A=(a_{i\bar j}).  \nonumber
 	\end{aligned}
 \end{equation}
 The linearized operator of equation \eqref{MA-n-1} at $u$ is given by
 \begin{equation}
 	\begin{aligned}
 		\mathcal{L}v = F^{i\bar j}(\mathfrak{\tilde{g}}[u])(v_{i\bar j}+W_{i\bar j}^k v_k+W_{i\bar j}^{\bar k}v_{\bar k}).  \nonumber
 	\end{aligned}
 \end{equation}
 where $W_{i\bar j}^k$ and $W_{i\bar j}^{\bar k}$ are defined in \eqref{w-rewrite1}. For simplicity, we denote
 $$F^{i\bar j}=F^{i\bar j}(\mathfrak{\tilde{g}}[u]).$$
 
 We derive quantitative boundary estimates for tangential-normal derivatives by using barrier functions. 
 This type of construction of barrier functions  
 goes back at least to \cite{Guan1993Spruck,Guan1998The}.  
 We shall point out that the constants in proof of quantitative boundary estimates,
 such as $C$, $C_{\Phi}$,  $C_1$, $C_1'$, $C_2$, $A_1$,  $A_2$, $A_3,$   etc, 
 depend on  neither  $|\nabla u|$ nor  $(\inf_M\phi)^{-1}$, nor $|\nabla \psi|$.
 
 By direct calculations, one derives  
 \begin{equation}
 	\begin{aligned}
 		\,&   u_{x_{k} l}=u_{l x_{k}}+T^{p}_{kl}u_{p}, \,&
 		u_{y_{k} l}=u_{l y_{k}}+\sqrt{-1}T^{p}_{kl}u_{p},  \nonumber
 	\end{aligned}
 \end{equation}
 \begin{equation}
 	\begin{aligned}
 		\,& (u_{x_k})_{\bar j}=u_{x_k\bar j}+\overline{\Gamma_{kj}^l} u_{\bar l}, 
 		\,& (u_{y_k})_{\bar j}=u_{y_k\bar j}-{\sqrt{-1}}\overline{\Gamma_{kj}^l} u_{\bar l}, \nonumber
 	\end{aligned}
 \end{equation}
 \begin{equation}
 	\begin{aligned}
 		(u_{x_k})_{i\bar j}=
 		u_{x_ki\bar j}+\Gamma_{ik}^lu_{l\bar j}+\overline{\Gamma_{jk}^l} u_{i\bar l}-g^{l\bar m}R_{i\bar j k\bar m}u_l,  \nonumber
 	\end{aligned}
 \end{equation}
 \begin{equation}
 	\begin{aligned}
 		(u_{y_k})_{i\bar j}=
 		u_{y_ki\bar j}+\sqrt{-1}(\Gamma_{ik}^l u_{l\bar j}-\overline{\Gamma_{jk}^l} u_{i\bar l})-\sqrt{-1}g^{l\bar m}R_{i\bar j k\bar m}u_l,
 		\nonumber
 	\end{aligned}
 \end{equation}
 where 
 \[R_{i\bar jk\bar l}= -\frac{\partial^2 g_{k\bar l}}{\partial z_i \partial\bar z_j}
 + g^{p\bar q}\frac{\partial g_{k\bar q}}{\partial z_i}\frac{\partial g_{p\bar l}}{\partial \bar z_j}.\]
 As a consequence
 \begin{equation}
 	\label{linear-1}
 	\begin{aligned}
 		\mathcal{L}(\pm u_{t_{\alpha}}) \geq \pm\psi_{t_{\alpha}}
 		-C(1+|\nabla u|)\sum_{i=1}^n f_i -C\sum_{i=1}^n f_i |\lambda_i|. 
 	\end{aligned}
 \end{equation}

 We denote 
  $$b_{1}=1+\sup_{M} |\nabla u|^{2}.$$
 \begin{lemma}
 	\label{yuan-key0}
 	Given $x_0\in\partial M$.
 	Let $u$ be a $C^3$-smooth $(n-1)$-PSH solution to equation \eqref{MA-n-1}, and $\Phi$ is defined as
 	\begin{equation}
 		\label{Phi-def1}
 		\begin{aligned}
 			\Phi=\pm \mathcal{T}(u-\varphi)+\frac{\gamma}{\sqrt{b_1}}(u_{y_{n}}-\varphi_{y_{n}})^2 \mbox{ } \mbox{ in } \Omega_\delta.
 		\end{aligned}
 	\end{equation}
 	Then there is a positive constant $C_{\Phi}$ depending on
 	$|\varphi|_{C^{3}(\bar M)}$, $|\chi|_{C^{1}(\bar M)}$
 	and other known data 
 	such that
 	\begin{equation}
 		\label{yuan-1}
 		\begin{aligned}
 			\mathcal{L}\Phi \geq 
 			-C_{\Phi}  \sqrt{b_1}  \sum_{i=1}^n f_i  - C_{\Phi}  \sum_{i=1}^n f_i|\lambda_i|
 			-C_\Phi |\nabla\psi| 
 		\mbox{ } \mbox{ } \mbox{ on } \Omega_{\delta} \nonumber
 		\end{aligned}
 	\end{equation}
 	for some small positive constant $\delta$.
 	In particular, if  $\partial M$ is holomorphically flat and $\varphi\equiv \mathrm{constant}$
 	then $C_{\Phi}$ depends on $|\chi|_{C^{1}(\bar M)}$
 	and other known data.

 \end{lemma}
 
 \begin{proof}
 	Together with 
 	\eqref{linear-1} and Cauchy-Schwarz inequality, one can compute and obtain 
 	\begin{equation}
 		\label{bdy-g1}
 		\begin{aligned}
 			\mathcal{L}(\pm\mathcal{T}u)
 			\geq \,&
 			-C\sqrt{b_1}\sum_{i=1}^n f_i  
 			- C\sum_{i=1}^n f_i|\lambda_i|
 			-\frac{\gamma}{\sqrt{b_1}} F^{i\bar j}u_{y_n i}u_{y_n \bar j} \pm\mathcal{T} \psi, \nonumber
 		\end{aligned}
 	\end{equation}
 	\begin{equation}
 		\begin{aligned}
 			F^{i\bar j}(\widetilde{\eta})_i (u_{x_{n}})_{\bar j}
 			\leq \,& C\sum_{i=1}^nf_i|\lambda_i|+ \frac{1}{\sqrt{b_1}}F^{i\bar j}u_{y_n i} u_{y_n \bar j} +C\sqrt{b_1}\sum_{i=1}^n f_i,
 			\nonumber
 		\end{aligned}
 	\end{equation}
 	\begin{equation}
 		\label{yuan-hao2}
 		\begin{aligned}
 			\mathcal{L}((u_{y_{n}}-\varphi_{y_{n}})^2)
 			\geq 
 			F^{i\bar j}u_{y_n i}u_{y_n \bar j}
 			-C\left(1+|\nabla u|^2\right)   \sum_{i=1}^n f_i  
 			- C|\nabla u|\sum_{i=1}^n f_i|\lambda_i|  - C|\nabla\psi|(1+|\nabla u|). \nonumber
 		\end{aligned}
 	\end{equation}
 	Putting these inequalities together we complete the proof.
 \end{proof}
 
 To estimate the quantitative boundary estimates for mixed derivatives, 
 we should employ barrier function of the form
 \begin{equation}
 	\label{barrier1}
 	\begin{aligned}
 		v= (\underline{u}-u)
 		- t\sigma
 		+N\sigma^{2}  \mbox{ }  \mbox{  in  } \Omega_{\delta},
 	\end{aligned}
 \end{equation}
 where $t$, $N$ are positive constants to be determined.

 Let $\delta>0$ and $t>0$ be sufficiently small with $N\delta-t\leq 0,$ such that, in $\Omega_{\delta}$,    $v\leq 0$, $\sigma$ is $C^2$ and
 \begin{equation}
 	\label{bdy1}
 	\begin{aligned}
 		\frac{1}{4} \leq |\nabla \sigma|\leq 2,  \quad
 		|\mathcal{L}\sigma | \leq   C_2\sum_{i=1}^n f_i. 
 	\end{aligned}
 \end{equation}

 We construct the barrier function as follows:
 \begin{equation}
 	\label{Psi}
 	\begin{aligned} 
 		\widetilde{\Psi} =A_1 \sqrt{b_1}v -A_2 \sqrt{b_1} |z|^2 + \frac{1}{\sqrt{b_1}} \sum_{\tau<n}|(u-\varphi)_{\tau}|^2+ A_3 \Phi \mbox{ } \mbox{ in } \Omega_\delta.
 	\end{aligned}
 \end{equation}

 \begin{proof}
 	[Proof of Proposition \ref{mix-general}]

 	If $A_2\gg A_3\gg1$ then one has  $\widetilde{\Psi}\leq 0 \mbox{ on } \partial \Omega_\delta$, here we use  \eqref{bdr-t}.
 	Note $\widetilde{\Psi}(x_0)=0$.  
 	It suffices  to prove $$\mathcal{L}\widetilde{\Psi}\geq 0
 	 \mbox{ } \mbox{ on } \Omega_\delta,$$
 	which yields $\widetilde{\Psi}\leq 0 $ in $\Omega_{\delta}$, and then $(\nabla_\nu \widetilde{\Psi})(x_0)\leq 0$.

 	By a direct computation one has
 	\begin{equation}
 		\label{L-v}
 		\begin{aligned}
 			\mathcal{L}v\geq F^{i\bar j} (\mathfrak{\underline{\tilde{g}}}_{i\bar j}-\mathfrak{\tilde{g}}_{i\bar j})-C_2 |2N\sigma-t|\sum_{i=1}^n f_i +2N F^{i\bar j}\sigma_i \sigma_{\bar j}. \nonumber
 		\end{aligned}
 	\end{equation}
 	Applying \cite[Lemma 6.2]{CNS3},  with a certain order of $\underline{\lambda}$,
 	\begin{equation}
 		\begin{aligned}
 			F^{i\bar j}   \mathfrak{\underline{\tilde{g}}}_{i\bar j} 
 			\geq \sum_{i=1}^n f_i(\lambda) \underline{\lambda}_i=\sum_{i=1}^n f_i  \underline{\lambda}_i. \nonumber
 		\end{aligned}
 	\end{equation}
 	By \cite[Proposition 2.19]{Guan12a} there is an index $r$ so that 
 	\begin{equation}
 		\label{L-u-2}
 		\begin{aligned}\sum_{\tau<n} F^{i\bar j}\mathfrak{\tilde{g}}_{\bar\tau i}\mathfrak{\tilde{g}}_{\tau \bar j}\geq \frac{1}{4}\sum_{i\neq r} f_{i}\lambda_{i}^{2}. \nonumber
 		\end{aligned}
 	\end{equation}
 	
 	In what follows we denote  $\widetilde{u}=u-\varphi$.
 	By straightforward computations  
 	\begin{equation}
 		\begin{aligned}
 			\mathcal{L} \left(\sum_{\tau<n}|\widetilde{u}_{\tau}|^2 \right)
 			\geq \,&
 			\frac{1}{2}\sum_{\tau<n}F^{i\bar j} \mathfrak{\tilde{g}}_{\bar \tau i} \mathfrak{\tilde{g}}_{\tau \bar j}
 			-C_1'\sqrt{b_1} \sum_{i=1}^n f_{i}|\lambda_{i}|  -C_1' b_1\sum_{i=1}^n f_{i} -C_1'\sqrt{ b_1}|\nabla \psi| \\ 
 			\geq \,&
 			\frac{1}{8}\sum_{i\neq r} f_{i}\lambda_{i}^{2}
 			-C_1'\sqrt{b_1} \sum_{i=1}^n  f_{i}|\lambda_{i}|-C_1' b_1 \sum_{i=1}^n f_{i} -C_1'\sqrt{b_1}|\nabla \psi|. \nonumber
 		\end{aligned}
 	\end{equation}

 	We are going to deal with $\sum_{i=1}^n f_i |\lambda_i|:$
 	\begin{enumerate}
 		\item
 		$\sum_{i=1}^n f_i |\lambda_i| = \sum_{i=1}^n f_i\lambda_i -2\sum_{\lambda_i<0} f_{i} \lambda_i
 		=n-2\sum_{\lambda_i<0} f_{i} \lambda_i;$
 		\item $\sum_{i=1}^n f_i |\lambda_i| =2\sum_{\lambda_i\geq 0} f_i\lambda_i -\sum_{i=1}^n f_{i} \lambda_i =2\sum_{\lambda_i\geq 0} f_i\lambda_i -n.$ 
 	\end{enumerate}
 	In conclusion, combining with Cauchy-Schwarz inequality, we have
 	\begin{equation}
 		\label{flambda}
 		\begin{aligned}
 			\sum_{i=1}^n f_i |\lambda_i|
 			\leq  \frac{\epsilon}{8\sqrt{b_1}}\sum_{i\neq r} f_i\lambda_i^2 +\frac{8\sqrt{b_1}}{\epsilon}\sum_{i=1}^n f_i+n. \nonumber
 		\end{aligned}
 	\end{equation}
 	
 	Taking $\epsilon=\frac{1}{C_1'+A_3C_\Phi }$.
 	Putting the above inequalities together we have
 	\begin{equation}
 		\label{bdy-main-inequality}
 		\begin{aligned}
 			\mathcal{L}\widetilde{\Psi} \geq \,&
 			A_1 \sqrt{b_1} \sum_{i=1}^n f_{i}(\underline{\lambda}_i-\lambda_i)
 			+ 2A_1N \sqrt{b_1} F^{i\bar j}\sigma_i \sigma_{\bar j}
 			\\ \,&
 			-  \{ C_1'+ A_2+A_3C_\Phi  +A_1C_2  |2N\sigma-t|
 			+8(C_1'+A_3C_\Phi)^2   
 			\\
 			\,&
 			+n(C_1'+A_3C_\Phi)/{\sqrt{b_1}}\} \sqrt{b_1}\sum_{i=1}^n f_i
 			-(C_1' +A_3 C_\Phi)|\nabla \psi|.
 		\end{aligned}
 	\end{equation}

 	Let's take $\varepsilon=\inf_M \min_{i}\underline{\mu}_i$  as in Lemma \ref{type-lemma1}, let $\theta_0= \frac{2\varepsilon^{n}}{(2n)^{n+1}} e^{-\max_M \psi}$.
 	
 	{\bf Case 1}: Suppose  
 	\begin{equation}
 		\begin{aligned}
 			\min_j\mu_j\leq \frac{\varepsilon}{2n}. \nonumber
 		\end{aligned}
 	\end{equation}
 	By Lemma \ref{type-lemma1}
 	we have
 	\begin{equation}
 		\label{guan-key1}
 		\begin{aligned}
 			\sum_{i=1}^n f_{i}(\underline{\lambda}_i-\lambda_i)  \geq \frac{\varepsilon}{2(n-1)} \sum_{i=1}^n f_i.
 		\end{aligned}
 	\end{equation}
 	The bad term $-(C_1' +A_3 C_\Phi)|\nabla \psi|$ can be controlled according to \eqref{sumfi-3}.
 	In addition, we  can choose $\delta$ and $t$  small enough  such that
 	\begin{equation}
 		\label{yuanbd-11}
 		\begin{aligned}
 			|2N\delta-t| \leq \min\left\{\frac{\varepsilon}{8C_{2}}, \frac{\theta_0}{16C_2}\right\}.
 		\end{aligned}
 	\end{equation}
 	Taking $A_1\gg 1$ we can derive $$\mathcal{L}\widetilde{\Psi} \geq 0 \mbox{ on } \Omega_\delta.$$

 	{\bf Case 2}: If $\min_j\mu_j> \frac{\varepsilon}{2n}$, 
 	then $\max_i\mu_i \leq (2n)^{n-1}e^\psi/\varepsilon^{n-1}$.    By \eqref{yuan-37} 
 	we have
 	\begin{equation}
 		\label{2nd-case1}
 		\begin{aligned}
 			\,& 
 			f_{i} \geq   \theta_0\sum_{i=j}^n f_j, 
 			\,& 
 			\forall 1\leq i\leq n.
 		\end{aligned}
 	\end{equation}
 	

 	All the bad terms containing $\sum_{i=1}^n f_i$ in \eqref{bdy-main-inequality} can be controlled by 
 	\begin{equation}
 		\label{bbvvv}
 		\begin{aligned}
 			A_1N \sqrt{b_1} F^{i\bar j}\sigma_i \sigma_{\bar j} \geq \frac{A_1N \theta_0   \sqrt{b_1}}{16}\sum_{i=1}^n f_i  \mbox{ on } \Omega_\delta.
 		\end{aligned}
 	\end{equation}
 	On the other hand,   
 	the bad term  $-(C_1' +A_3 C_\Phi)|\nabla \psi|$ in last term of \eqref{bdy-main-inequality}
 	can be dominated by combining  \eqref{sumf_i1} with \eqref{bbvvv}.
 	Then  $\mathcal{L}(\widetilde{\Psi}) \geq 0 $ on $\Omega_\delta$, if one chooses $A_1N\gg 1$.
 \end{proof}
 
When the boundary is holomorphically flat  
 and the boundary data is a constant, we have a slightly delicate result.
 
 \begin{proposition}
 	\label{mix-Leviflat}
 	Let $(M,J,\omega)$ be a compact Hermitian manifold with  holomorphically flat boundary.
 	Suppose, in addition to 
 	\eqref{subsolution-MA-n-1}  
 	and $\phi^{{1}/{n}}\in C^1(\bar M)$, that the boundary data $\varphi$ is a constant.
 	Then
 	every strictly $(n-1)$-PSH solution $u\in C^3(M)\cap C^2(\bar M)$ of Dirichlet problem \eqref{MA-n-1} and \eqref{bdy-value-2} satisfies
 	\begin{equation}
 		|\mathfrak{\tilde{g}}(\xi_\alpha,J\bar\xi_n)|\leq C 
 		\left(1+\sup_{M}|\nabla u|\right), \nonumber
 	\end{equation}
 	where $C$ depends on 
 	$|\phi^{{1}/{n}}|_{C^{1}(\bar M)}$, $|\underline{u}|_{C^{2}(\bar M)}$, $|\nabla u|_{C^0(\partial M)}$,
 	$\partial M$ 
 	up to second derivatives
 	and other known data, but neither on $\sup_{M}|\nabla u|$ nor on $(\inf_M\phi)^{-1}$. 
 	
 \end{proposition}

 \subsection{Further results under appropriate conditions on $\phi^{{1}/{(n-1)}}$}
 
 As in \eqref{tildephi-1} we denote $\tilde{\phi}=\phi^{1/(n-1)}$.
 If $0\leq \tilde{\phi}\in C^{1,1}(\bar M)$ satisfies \eqref{keykey-phi}, then 
 \begin{equation}
 	\begin{aligned} 
 		\nabla\tilde{\phi}=0 \mbox{ at points $p\in\bar M$ where } \tilde{\phi}=0.
 	\end{aligned}
 \end{equation}
This implies that there exists a uniform positive constant $C$ such that
 \begin{equation}
 	\label{sqrt-condition1}
 	\begin{aligned}
 		|\nabla\tilde{\phi}|\leq C\sqrt{\tilde{\phi}}   \mbox{ in }\bar M.
 	\end{aligned}
 \end{equation}


 \subsubsection{Stability of condition \eqref{sqrt-condition1}}

 Fix  $0<\epsilon<1$. Suppose
 $\tilde{\phi}_\epsilon\in C^2(\bar M)$ and it satisfies
 \begin{equation} \begin{aligned}
 		|\tilde{\phi}_\epsilon-(\tilde{\phi}+\epsilon)|_{C^{1,1}(\bar M)}\leq \frac{\epsilon}{2}. \nonumber
 \end{aligned} \end{equation}
 Then 
 \begin{equation}
 	\begin{aligned}
 		\tilde{\phi}+\frac{\epsilon}{2}\leq \tilde{\phi}_\epsilon \leq \tilde{\phi}+\frac{3\epsilon}{2},   \nonumber
 	\end{aligned}
 \end{equation}
 \begin{equation}
 	\begin{aligned}
 		| \nabla\tilde{\phi}_\epsilon| \leq |\nabla\tilde{\phi}|+\frac{\epsilon}{2} \leq C\sqrt{\tilde{\phi}}+\frac{\epsilon}{2}\leq
 		(1+C)\sqrt{\tilde{\phi}+\frac{\epsilon}{2}}
 		\leq (1+C)\sqrt{\tilde{\phi}_\epsilon}. \nonumber
 	\end{aligned}
 \end{equation}

 \subsubsection{Quantitative boundary estimate revisited}
 
 \begin{proposition}
 	\label{mix-Leviflat--2}
 	Let $0<\tilde{\phi}=\phi^{{1}/{(n-1)}}\in C^{1}(\bar M)$ satisfy \eqref{keykey-phi}. Suppose the other assumptions in Proposition \ref{mix-general} hold.
 	Then we have the quantitative boundary estimates \eqref{quanti-mix-derivative-00}.  
	In addition  \eqref{keykey-phi} can be removed when $M=X\times S$.
 \end{proposition}

 \begin{proof}[Sketch of proof]
 	Together with \eqref{sumf_i1}, \eqref{sqrt-condition1} implies   
 	\begin{equation}
 		\label{j1}
 		\begin{aligned}
 			|\nabla\psi|=\frac{(n-1)|\nabla\phi^{1/(n-1)}|}{\phi^{{1}/{(n-1)}}} \leq \frac{C}{\phi^{1/2(n-1)}} \leq 
 			\frac{C}{\phi^{1/n}} \leq 
 			\frac{C}{n}\sum_{i=1}^n f_i.  \nonumber
 		\end{aligned}
 	\end{equation}
 	
 	When $M=X\times S$ we always have the following inequality even if \eqref{keykey-phi} does not hold,
 	\begin{equation}
 		\label{sqrt-condition2}
 		\begin{aligned}
 			|\nabla_\xi\phi^{\frac{1}{n-1}}|\leq C\phi^{{1}/{2(n-1)}}
 		\end{aligned}
 	\end{equation}
 	where $\xi\in T^{1,0}_X$ and $|\xi|=1$.
 	On the other hand, 
 	the tangential operator is 
 	$$\mathcal{T}= \frac{\partial}{\partial z_{\alpha}}, \mbox{ }  \frac{\partial}{\partial \bar z_{\alpha}}, \mbox{ } 1\leq\alpha\leq n-1
 	$$ where  $z'=(z_1,\cdots z_{n-1})$ is local holomorphic coordinate of $X$. And
 	the barrier function is 
 	\begin{equation}
 		\label{Psi2}
 		\begin{aligned} 
 			\widetilde{\Psi} =A_1 \sqrt{b_1}v -A_2 \sqrt{b_1} |z|^2 + \frac{1}{\sqrt{b_1}} \sum_{\tau<n}|(u-\varphi)_{\tau}|^2+ A_3 \pm\mathcal{T}(u-\varphi). \nonumber 
 		\end{aligned}
 	\end{equation}
 	Similar to \eqref{bdy-main-inequality} we obtain
 	\begin{equation}
 		\begin{aligned}
 			\mathcal{L}\widetilde{\Psi} \geq \,&
 			A_1 \sqrt{b_1} \sum_{i=1}^n f_{i}(\underline{\lambda}_i-\lambda_i)
 			+\frac{1}{8\sqrt{b_1}} \sum_{i\neq r} f_i\lambda_i^2 
 			\\\,&
 			+ 2A_1N \sqrt{b_1} F^{i\bar j}\sigma_i \sigma_{\bar j}
 			-C \sqrt{b_1}\sum_{i=1}^n f_i-C\sum_{\tau=1}^{n-1}|\psi_\tau|. \nonumber
 		\end{aligned}
 	\end{equation}
 	This completes the proof by using \eqref{sqrt-condition2}.

 \end{proof}

 \subsubsection{Global second estimate revisited}
 
One can check Sz\'ekelyhidi-Tosatti-Weinkove's proof of second estimate works under such weaker assumptions
 on $\phi^{1/(n-1)}$.

 \begin{proposition}
 	\label{remark2.4}
 	
 	Let $0<\tilde{\phi}=\phi^{{1}/{(n-1)}}\in C^2(\bar M)$ satisfy \eqref{sqrt-condition1}, we assume that there is a $(n-1)$-PSH function 
 	$\underline{u}\in C^2(\bar M)$. Then there is a positive constant $C$ depending on 
 	$|\phi^{{1}/{(n-1)}}|_{C^2(\bar M)}$, $(\inf_{M}\min_i\underline{\mu}_i)^{-1}$, $|\underline{u}|_{C^2(\bar M)}$ and other known data (but neither on $(\inf_{M}\phi)^{-1}$ nor on $\sup_M|\nabla u|$)
 	such that  \eqref{sec-estimate-quar1} and \eqref{STW-estimate} hold. 
 \end{proposition}
 
 \begin{proof}[Sketch of proof]
 	Let  $\mu_i=\sum_{j\neq i}\lambda_j$, $\psi=\log\phi+n\log(n-1)$.  
 	Generalized Newton-Maclaurin inequality gives
 	\begin{equation}
 		\label{Newton-Maclaurin1}
 		\begin{aligned}
 			\sum_{i=1}^n \frac{1}{\mu_i} \geq n^{\frac{n-2}{n-1}}\left(\frac{\sum_{i=1}^n \mu_i}{\mu_1\cdots\mu_n}\right)^{1/(n-1)}. 
 		\end{aligned}
 	\end{equation}

 	
 	By straightforward computations we get
 	\begin{equation}
 		\begin{aligned}
 			(\log\phi)_{i\bar j}=\frac{\phi_{i\bar j}}{\phi}-\frac{\phi_i\phi_{\bar j}}{\phi^2},  \nonumber
 		\end{aligned}
 	\end{equation} 
 	\begin{equation}
 		\begin{aligned}
 			(\phi^{1/(n-1)})_{i\bar j}=
 			\frac{1}{n-1} \phi^{1/(n-1)} \left(\frac{\phi_{i\bar j}}{\phi}-\frac{n-2}{n-1} \frac{\phi_i \phi_{\bar j}}{\phi^2} \right). \nonumber
 		\end{aligned}
 	\end{equation}
 	
 	Combining with \eqref{Newton-Maclaurin1} and concavity of $f$,
 	for $$\sum_{i=1}^n\log\mu_i=\log\phi+n\log(n-1),$$
 	we obtain
 	\begin{equation}
 		\begin{aligned}
 			\sum_{i=1}^n f_i\geq n^{\frac{n-2}{n-1}}
 			\left(\sum_{i=1}^n\lambda_i\right)^{\frac{1}{n-1}}\phi^{-\frac{1}{n-1}},
 		\end{aligned}
 	\end{equation}
 	\begin{equation}
 		\label{key1-phi}
 		\begin{aligned}
 			|\psi_k|=|(\log\phi)_k|=(n-1)\phi^{-1/(n-1)}|(\phi^{1/(n-1)})_k|,
 		\end{aligned}
 	\end{equation}
 	\begin{equation}
 		\label{key2-phi}
 		\begin{aligned}
 			\psi_{k\bar k}=(\log\phi)_{k\bar k} 
 			=
 			(n-1)\phi^{-\frac{1}{n-1}}(\phi^{{1}/{(n-1)}})_{k\bar k}-\frac{1}{n-1}\frac{|\phi_k|^2}{\phi^2}.
 		\end{aligned}
 	\end{equation}
 	By \eqref{sqrt-condition1} again, 
 	\begin{equation}
 		\label{nnn}
 		\begin{aligned}
 			\frac{|\phi_k|^2}{\phi^2} \leq \frac{|\nabla\phi|^2}{\phi^2}=(n-1)^2\frac{|\nabla\phi^{{1}/{(n-1)}}|^2}{\phi^{\frac{2}{n-1}}}\leq C\phi^{-\frac{1}{n-1}}.
 		\end{aligned}
 	\end{equation}
 	Thus 
 	\begin{equation} \label{key3-phi} \begin{aligned}
 			|\psi_k|,   -\psi_{k\bar k}  \leq C\left(\sum_{i=1}^n \lambda_i\right)^{-\frac{1}{n-1}}\sum_{i=1}^n f_i. \nonumber
 	\end{aligned} \end{equation}
 	
With those conditions at hand,  
one can prove the second order 
 	estimates 
 	\eqref{STW-estimate}, following the original proof of Sz\`ekelyhidi-Tosatti-Weinkove almost words by words. 
 	
 	
 \end{proof}

\begin{remark}
When $\partial M=\emptyset$, each $0\leq \tilde{\phi}=\phi^{1/(n-1)}\in C^{2}(M)$ satisfies  \eqref{sqrt-condition1}. Thus one can improve Sz\'ekelyhidi-Tosatti-Weinkove's second estimate on closed Hermitian manifolds. 
\end{remark}

 \section{The Dirichlet problem  with less regular boundary and boundary data} 
 \label{sec5}

 The purpose of this section is to investigate the equations on complex manifolds with  less regular boundary.\renewcommand{\thefootnote}{\fnsymbol{footnote}}\footnote{We emphasize that   the geometric quantities of $(M,\omega)$ (the curvature $R_{i\bar j k\bar l}$ and  torsion $T^k_{ij}$) keep bounded as approximating to $\partial M$, and all derivatives of ${\chi}_{i\bar j}$ has continues extensions to $\bar M$,
 	whenever $M$ has less regularity boundary. 
 }

 We first state an observation. 
 \begin{lemma}
 	\label{obser-subsolution}
 	Let $Z = \frac{1}{(n-1)!}*\mathfrak{Re}(\sqrt{-1}\partial u\wedge \overline{\partial}(\omega^{n-2}))$ be as in \eqref{MA-n-1}.
 	For any $C^1$-smooth function $v$ on $\bar S$,
 	\begin{equation}
 		\begin{aligned}
 			Z[v]=0.  \nonumber
 		\end{aligned}
 	\end{equation}
 \end{lemma}
 \begin{proof}
 	Note that $\omega_X$ is balanced, and $v$ is a function on $\bar S$, we see
 	$ \overline{\partial}\omega_X^{n-2}=0$ and $\partial v\wedge \omega_S=0$; thus
 	$\partial v\wedge \overline{\partial}\omega^{n-2}=0$.
 \end{proof}
 

 A somewhat remarkable fact to us is that the regularity assumptions on boundary and boundary data can be further weakened under certain assumptions. 
 The motivation is mainly based on the estimates which state that,
 if $\partial M$ is  {holomorphically flat}  and  boundary value is a constant, 
 then the constant in quantitative boundary estimate \eqref{bdy-sec-estimate-quar1}  
 depends only on $\partial M$ up to second derivatives and other known data (see Propositions \ref{mix-Leviflat} or \ref{mix-Leviflat--2} and \ref{proposition-quar-yuan2}). 
 Besides, we can use a result due to Silvestre-Sirakov \cite{Silvestre2014Sirakov} to derive
 the $C^{2,\alpha}$ boundary regularity with only assuming $C^{2,\beta}$ boundary.

 As consequences, we can prove Theorems \ref{thm2-volumeform}-\ref{thm2-volumeform-2} and the following theorem.
 \begin{theorem}
 	\label{thm3-volumeform}
 	
 	Let $(M,J,\omega)$ be a product as  in \eqref{product-1}, 
 	and we assume $\partial S\in C^{2,\beta}$ for some $0<\beta<1$. 
 		  If $0\leq\phi^{1/n}\in C^{1,1}(\bar M)$ then the Dirichlet problem \eqref{MA-n-1} with homogeneous boundary data admits a $C^{1,\alpha}$-smooth 
 		$(n-1)$-PSH solution with $\forall 0<\alpha<1$ and $\Delta u\in {L^\infty(\bar M)}$  in the weak sense.

 \end{theorem}
 
 \begin{proof}
 	[Sketch of proof of Theorems \ref{thm2-volumeform}, \ref{thm2-volumeform-2} and \ref{thm3-volumeform}]
 	
 	It only requires to consider the nondegenerate case: 
 	\[\phi>\delta_0\mbox{ in } \bar M \mbox{ for some } \delta_0>0.\]
 	
 	Let $h$ be the solution to \eqref{possion-def}, and we 
 	denote $\underline{u}= th$. For $t\gg1$ we have
 	\begin{equation}
 		\begin{aligned}
 			\left(\tilde{\chi}+\frac{1}{n-1}(\Delta \underline{u}\omega-\sqrt{-1}\partial\overline{\partial} \underline{u})+
 			Z[\underline{u}]\right)^n\geq  (\phi+\delta_1)\omega^n \mbox{ in } M
 		\end{aligned}
 	\end{equation}
 	for some $\delta_1>0$.  In fact $Z[\underline{u}]=0$ by Lemma \ref{obser-subsolution}.
 	Since 
 	we know that 
 	\[h\in C^\infty(S)\cap C^{2,\beta}(\bar S), \mbox{  } h|_S<0, \mbox{  } \frac{\partial h}{\partial\nu}|_{\partial S}<0,\]
 	we can carefully choose $\{\alpha_k\}$ with $\alpha_{k}\rightarrow 0^+$ as $k\rightarrow +\infty$ then use a sequence of level sets 
 	$\{h=-\alpha_k\}$
 	to enclose a smooth
 	Riemann surface, denoted by $S_{k}$, such that 
 	$\cup S_{k}=S$ and $\partial S_{k}$ converge to $\partial S$ in the norm of $C^{2,\beta}$. 
 	Let's denote $M_{k}=X\times S_{k}$. 
 	For any $k\geq 1$, there exists a $\phi^{(k)}\in C^\infty(\bar M_k)$ 
 	such that
 	\[|\phi-\phi^{(k)}|_{C^2(\bar M_k)}\leq \beta_k \rightarrow0^+ \mbox{ as } k\rightarrow+\infty.\]
 	For $k\gg1$ we have $\beta_k<\min\{\delta_0,\delta_1\}$, then
 	\begin{equation}
 		\label{subsolution-newregularity}
 		\begin{aligned}
 			\left(\tilde{\chi}+\frac{1}{n-1}(\Delta \underline{u}\omega-\sqrt{-1}\partial\overline{\partial}\underline{u})+Z[\underline{u}]\right)^n
 			\geq \,&\phi^{(k)} \omega^n  \,& \mbox{ in } M_k, 
 			\\ \underline{u}=\,& -t\alpha_k  \,& \mbox{ on }\partial M_k.  \nonumber
 		\end{aligned}
 	\end{equation}

 	According to Theorem \ref{thm0-n-1-yuan3} 
 	we have a unique smooth $(n-1)$-PSH function $u^{(k)}\in C^{\infty}(\bar M_{k})$ to solve
 	\begin{equation}
 		\label{solution-newregularity}
 		\begin{aligned}
 			\left(\tilde{\chi}+\frac{1}{n-1}(\Delta u^{(k)} \omega-\sqrt{-1}\partial\overline{\partial}u^{(k)})+Z[u^{(k)}]\right)^n=\,& \phi^{(k)} \omega^n
 			\,& \mbox{ in } M_k, 
 			\\  u^{(k)} =\,& -t\alpha_k  \,& \mbox{ on } \partial M_k. \nonumber
 		\end{aligned}
 	\end{equation}
 	Moreover,  Propositions \ref{remark2.4}, \ref{proposition-quar-yuan2} and \ref{mix-Leviflat} or \ref{mix-Leviflat--2}, and 
 	Theorem \ref{STW-estimate1}
 	yield
 	\begin{equation}
 		\label{uniform-00}
 		\begin{aligned}
 			\sup_{M_{k}}\Delta u^{(k)}\leq C_k (1+\sup_{M_{k}}|\nabla u^{(k)}|^2), \nonumber
 		\end{aligned}
 	\end{equation}
 	where $C_k$ depends on $|\nabla u^{(k)}|_{C^0(\partial M_{k})}$, $|u^{(k)}|_{C^0(M_{k})}$, $|\underline{u}|_{C^{2}(M_{k})}$,  
 	$|(\phi^{(k)})^{1/n}|_{C^{2}(M_{k})}$ ($|(\phi^{(k)})^{1/(n-1)}|_{C^{2}(M_{k})}$ if \eqref{keykey-phi} holds),  $\partial M_{k}$ up to second order derivatives  and other known data (but not on $\inf_M \phi^{(k)}$).

 	If there is a uniform positive constant $C$ depending not on $k$, such that
 	\begin{equation}
 		\label{uniform-c0-c1}
 		\begin{aligned}
 			|u^{(k)}|_{C^0(M_{k})}+\sup_{\partial M_{k}}|\nabla u^{(k)}|\leq C,
 		\end{aligned}
 	\end{equation}
 	then
 	\begin{equation}
 		\label{uniform-00}
 		\begin{aligned}
 			\sup_{M_{k}}\Delta u^{(k)}\leq C' (1+\sup_{M_{k}}|\nabla u^{(k)}|^2) \mbox{ independent of } k. \nonumber
 		\end{aligned}
 	\end{equation}
 	Thus we have $|u|_{C^2(M_{k})}\leq C$ depending not on $k$ (here we  use blow up argument to derive gradient estimate).  
 	Finally, we are able to apply Silvestre-Sirakov's \cite{Silvestre2014Sirakov} result to 
 	derive $C^{2,\alpha'}$ estimates on the boundary, while
 	the convergence of  $\partial M_{k}$ in the norm $C^{2,\beta}$ 
 	allows  us to take a limit ($\alpha'$ can be uniformly chosen).
 	
 	It only requires to prove \eqref{uniform-c0-c1}.  Let $w^{(k)}$ be the solution of
 	\begin{equation}
 		\label{supersolution-k}
 		\begin{aligned}
 			\,&\Delta w^{(k)}+\mathrm{tr}_\omega \tilde{\chi}+\mathrm{tr}_\omega(Z[w^{(k)}])=0 \mbox{ in } M_{k}, \,& w^{(k)}=-t\alpha_k \mbox{ on } \partial M_{k}. \nonumber
 		\end{aligned}
 	\end{equation}
 	
 	By maximum principle and the boundary value condition, 
 	we have 
 	\begin{equation}
 		\label{approxi-boundary1}
 		\begin{aligned}
 			\underline{u} \leq u^{(k)}\leq w^{(k)} \mbox{ in } M_{k}, \mbox{  } 
 			\frac{\partial \underline{u}}{\partial \nu} \leq  \frac{\partial u^{(k)}}{\partial \nu} \leq \frac{\partial w^{(k)}}{\partial \nu}   \mbox{ on } \partial M_{k}.
 		\end{aligned}
 	\end{equation}

 	It remains to prove  
 	\begin{equation}
 		\label{key-proof-newregularity1}
 		\begin{aligned}
 			\sup_{M_{k}} w^{(k)} + \sup_{\partial M_{k}} \frac{\partial w^{(k)}}{\partial \nu}  \leq C \mbox{ independent of } k.
 		\end{aligned}
 	\end{equation}
 	For $t\gg1$,
 	\begin{equation}
 		\begin{aligned}
 			\Delta(-\underline{u}-2t\alpha_k)+\mathrm{tr}_\omega \tilde{\chi}+\mathrm{tr}_\omega (Z[-\underline{u}])=-t+\mathrm{tr}_\omega\tilde{\chi}\leq0. \nonumber
 		\end{aligned}
 	\end{equation}
 	Here Lemma \ref{obser-subsolution} implies $Z[-\underline{u}]=0$.  
 	Applying comparison principle, 
 	\begin{equation}
 		\begin{aligned}
 			\,& w^{(k)}\leq -\underline{u}-2t\alpha_k  \mbox{ in } M_{k}, \,&  \frac{\partial w^{(k)}}{\partial \nu} \leq -\frac{\partial
 				\underline{u}}{\partial \nu} \mbox{ on } \partial M_{k} \nonumber
 		\end{aligned}
 	\end{equation}
 	as required. 
 	We then obtain a $C^{2,\alpha}$-smooth $(n-1)$-PSH function to solve 
 	\begin{equation}
 		\label{approx-equ-homogeneous1}
 		\begin{aligned}
 			\left(\tilde{\chi}+\frac{1}{n-1}(\Delta u\omega-\sqrt{-1}\partial\overline{\partial}u)+Z\right)^n= \phi \omega^n \mbox{ in } M, 
 			\mbox{  } u = 0\mbox{ on } \partial M.  \nonumber
 		\end{aligned}
 	\end{equation}

 \end{proof}

 \section{Further discussion on more general equations}
 \label{sec6}
 
 Our method works for more general equations generated by  smooth symmetric functions $f$ defined on $\Gamma\subset\mathbb{R}^n$, dating to  the work of Caffarelli-Nirenberg-Spruck \cite{CNS3}. We consider the Dirichlet problem
 \begin{equation}
 	\label{mainequ-gauduchon-general**}
 	\begin{aligned}
 		\,& f(\lambda(*\Phi[u]))=\psi \mbox{ in } M,     \,&
 		u=\varphi \mbox{ on }   \partial M 
 	\end{aligned}
 \end{equation}
 where $*\Phi[u]=\chi+\Delta u \omega-\sqrt{-1}\partial\overline{\partial}u+\varrho Z[u],$  
 and $\varrho$ is a smooth function, i.e.
 $$\Phi[u] =*\chi+\frac{1}{(n-2)!}\sqrt{-1}\partial\overline{\partial}u\wedge\omega^{n-2}+\frac{\varrho}{(n-1)!}\mathfrak{Re}(\sqrt{-1}\partial u\wedge \overline{ \partial}\omega^{n-2}).$$
 Here $\Gamma$ is an open symmetric convex cone containing positive cone
 \[\Gamma_n:=\{\lambda\in \mathbb{R}^n: \mbox{ each component } \lambda_i>0\}\subseteq\Gamma\]
 with vertex at the origin  and with  boundary $\partial \Gamma\neq \emptyset$.
 In addition we assume
 \begin{equation}
 	\label{elliptic}
 	\begin{aligned}
 		\,& f_i(\lambda):=\frac{\partial f}{\partial \lambda_{i}}(\lambda)> 0  \mbox{ in } \Gamma,\,&  \forall 1\leq i\leq n,
 	\end{aligned}
 \end{equation}
 \begin{equation}\label{concave} \begin{aligned}
 		f \mbox{ is  concave in } \Gamma,
 \end{aligned} \end{equation}
 \begin{equation}
 	\label{addistruc}
 	\begin{aligned}
 		\mbox{For any } \lambda\in\Gamma, \mbox{  } 
 		\lim_{t\rightarrow+\infty} f(t\lambda)>-\infty,
 	\end{aligned}
 \end{equation}
 \begin{equation}
 	\label{unbounded-2}
 	\begin{aligned}
 		\,&  \lim_{t\rightarrow+\infty}  f(\lambda_1+t,\cdots,\lambda_{n-1}+t,\lambda_n)=\sup_\Gamma f, \,& \forall \lambda\in\Gamma.
 	\end{aligned}
 \end{equation}
 The equation \eqref{mainequ-gauduchon-general**} is nondegenerate when the right-hand side satisfies
 \begin{equation}
 	\label{nondegenerate}
 	\begin{aligned}
 		\inf_{M} \psi >\sup_{\partial \Gamma} f. 
 	\end{aligned}
 \end{equation}
Also it is called a degenerate equation if $\inf_{M} \psi=\sup_{\partial \Gamma} f$ and $f\in C^\infty(\infty)\cap C(\Gamma)$, where 
 $$\sup_{\partial \Gamma}f :=\sup_{\lambda_{0}\in \partial \Gamma } \limsup_{\lambda\rightarrow \lambda_{0}}f(\lambda).$$

 \begin{theorem}
 	Let $(M,J,\omega)$ be a compact Hermitian manifold  with smooth boundary.
 	Suppose, in addition to 
 	\eqref{elliptic}-\eqref{nondegenerate}, that $\varphi$, $\psi$ are all smooth and there is a $C^{2,1}$ function such that
 	\begin{equation}
 		\begin{aligned}
 			f(\lambda(*\Phi[\underline{u}]))\geq\psi, \mbox{ } \lambda(*\Phi[\underline{u}])\in\Gamma \mbox{ in } \bar M,     \quad
 			\underline{u}=\varphi \mbox{ on }   \partial M. 
 		\end{aligned}
 	\end{equation}
 	Then Dirichlet problem \eqref{mainequ-gauduchon-general**} admits a unique smooth solution with $\lambda(*\Phi[u])\in\Gamma$ in $\bar M$. 
 \end{theorem}

 \begin{remark}
 	Let's denote $$\Gamma_{\mathbb{R}^1}^\infty:=\{t\in\mathbb{R}: (R,\cdots,R,t)\in\Gamma \mbox{ for some } R>0\}.$$
 	The case $\Gamma\neq\Gamma_n$ is relatively simple,  since 
 	$\Gamma_{\mathbb{R}^1}^\infty=\mathbb{R}$ in this case.
 	While for the case $\Gamma=\Gamma_n$, the proof is almost parallel to that of Monge-Amp\`ere equation for $(n-1)$-PSH functions.  Moreover, we can solve the degenerate Dirichlet problem, when the Levi form satisfies  $$-(\kappa_1+\cdots+\kappa_{n-1})\in \overline{\Gamma_{\mathbb{R}^1}^\infty}.$$ 
 \end{remark}
 
\begin{remark}
	
	In the presence of	\eqref{elliptic}-\eqref{addistruc},  \eqref{unbounded-2} is automatically satisfied if $\Gamma\neq\Gamma_n$.  

\end{remark} 
 
 \begin{remark}
 	On the product $M=X\times S$ with closed balanced factor, 
 	we can use the solution of \eqref{possion-def} to construct the strictly subsolutions.
 \end{remark}

 \begin{appendix}

 	\section{A quantitative lemma} 
 	\label{appendix1}

 The  following lemma  proposed in earlier works \cite{yuan2017}\renewcommand{\thefootnote}{\fnsymbol{footnote}}\footnote{The results in \cite{yuan2017}  
 		were removed to \cite{yuan-V}. More precisely, the paper  \cite{yuan-V} is essentially extracted from 
 	\cite{yuan2017}, 
	and the first parts of   [arXiv:2001.09238] and   [arXiv:2106.14837].}, 
 	 is a key ingredient in proof of Proposition 
 	\ref{proposition-quar-yuan2}.

 	\begin{lemma}
 		[\cite{yuan2017,yuan-V}]
 		\label{yuan's-quantitative-lemma}
 		Let $A$ be an $n\times n$ Hermitian matrix
 		\begin{equation}\label{matrix3}\left(\begin{matrix}
 				d_1&&  &&a_{1}\\ &d_2&& &a_2\\&&\ddots&&\vdots \\ && &  d_{n-1}& a_{n-1}\\
 				\bar a_1&\bar a_2&\cdots& \bar a_{n-1}& \mathrm{{\bf a}} 
 			\end{matrix}\right)\end{equation}
 		with $d_1,\cdots, d_{n-1}, a_1,\cdots, a_{n-1}$ fixed, and with $\mathrm{{\bf a}}$ variable.
 		Denote $\lambda=(\lambda_1,\cdots, \lambda_n)$ by   the eigenvalues of $A$.
 		Let $\epsilon>0$ be a fixed constant.
 		Suppose that  the parameter $\mathrm{{\bf a}}$ in $A$ satisfies  the quadratic
 		growth condition  
 		\begin{equation}
 			\begin{aligned}
 				\label{guanjian1-yuan}
 				\mathrm{{\bf a}}\geq \frac{2n-3}{\epsilon}\sum_{i=1}^{n-1}|a_i|^2 +(n-1)\sum_{i=1}^{n-1} |d_i|+ \frac{(n-2)\epsilon}{2n-3}.
 			\end{aligned}
 		\end{equation}
 		Then the eigenvalues 
 		(possibly with a proper order) behave like
 		\begin{equation}
 		\begin{aligned}
 			d_{\alpha}-\epsilon 	\,& < 
 			\lambda_{\alpha} < d_{\alpha}+\epsilon, \mbox{  } \forall 1\leq \alpha\leq n-1, \\ \nonumber
 			\mathrm{{\bf a}} 	\,& \leq \lambda_{n}
 			< \mathrm{{\bf a}}+(n-1)\epsilon. \nonumber
 		\end{aligned}
 	\end{equation}
 	\end{lemma}
 	
 	
 	For convenience we give the proof of Lemma \ref{yuan's-quantitative-lemma} in this appendix. 
 	We start with the case of $n=2$. 
 	For $n=2$, the eigenvalues of $\mathrm{A}$ are
 	$$\lambda_{1}=\frac{\mathrm{{\bf a}}+d_1- \sqrt{(\mathrm{{\bf a}}-d_1)^2+4|a_1|^2}}{2} 
 \mbox{ and } \lambda_2=\frac{\mathrm{{\bf a}}+d_1+\sqrt{(\mathrm{{\bf a}}-d_1)^2+4|a_1|^2}}{2}.$$
 	We can assume $a_1\neq 0$; otherwise we are done.
 	If $\mathrm{{\bf a}} \geq \frac{|a_1|^2}{ \epsilon}+ d_1$ then one has
 	\begin{equation}
 		\begin{aligned}
 			0\leq d_1- \lambda_1 =\lambda_2-\mathrm{{\bf a}}
 			= \frac{2|a_1|^2}{\sqrt{ (\mathrm{{\bf a}}-d_1)^2+4|a_1|^2 } +(\mathrm{{\bf a}}-d_1)}
 			< \frac{|a_1|^2}{\mathrm{{\bf a}}-d_1 } \leq \epsilon.   \nonumber
 		\end{aligned}
 	\end{equation}

 	The following lemma enables us  to count  the eigenvalues near the diagonal elements
 	via a deformation argument.
 	It is an essential  ingredient in the proof of  Lemma \ref{yuan's-quantitative-lemma}   for general $n$.
 	\begin{lemma}
 		[\cite{yuan2017,yuan-V}]
 		\label{refinement}
 		Let $\mathrm{A}$ be an $n\times n$  Hermitian matrix
 		\begin{equation}
 			\label{matrix2}
 			\left(
 			\begin{matrix}
 				d_1&&  &&a_{1}\\
 				&d_2&& &a_2\\
 				&&\ddots&&\vdots \\
 				&& &  d_{n-1}& a_{n-1}\\
 				\bar a_1&\bar a_2&\cdots& \bar a_{n-1}& \mathrm{{\bf a}} \nonumber
 			\end{matrix}
 			\right)
 		\end{equation}
 		with $d_1,\cdots, d_{n-1}, a_1,\cdots, a_{n-1}$ fixed, and with $\mathrm{{\bf a}}$ variable.
 		Denote
 		$\lambda_1,\cdots, \lambda_n$ by the eigenvalues of $\mathrm{A}$ with the order
 		$\lambda_1\leq \lambda_2 \leq\cdots \leq \lambda_n$.
 		Fix a positive constant $\epsilon$.
 		Suppose that the parameter $\mathrm{{\bf a}}$ in the matrix $\mathrm{A}$ 
 		satisfies  the following quadratic growth condition
 		\begin{equation}
 			\label{guanjian2}
 			\begin{aligned}
 				\mathrm{{\bf a}} \geq \frac{1}{\epsilon}\sum_{i=1}^{n-1} |a_i|^2+\sum_{i=1}^{n-1}  [d_i+ (n-2) |d_i|]+ (n-2)\epsilon.
 			\end{aligned}
 		\end{equation}
 		Then for any $\lambda_{\alpha}$ $(1\leq \alpha\leq n-1)$ there exists $d_{i_{\alpha}}$
 		with lower index $1\leq i_{\alpha}\leq n-1$ such that
 		\begin{equation}
 			\label{meishi}
 			\begin{aligned}
 				|\lambda_{\alpha}-d_{i_{\alpha}}|<\epsilon,
 			\end{aligned}
 		\end{equation}
 		\begin{equation}
 			\label{mei-23-shi}
 			0\leq \lambda_{n}-\mathrm{{\bf a}} <(n-1)\epsilon + |\sum_{\alpha=1}^{n-1}(d_{\alpha}-d_{i_{\alpha}})|.
 		\end{equation}
 	\end{lemma}

 	\begin{proof}
 		Without loss of generality, we assume $\sum_{i=1}^{n-1} |a_i|^2>0$ and  $n\geq 3$
 		(otherwise we are done, since $\mathrm{A}$ is diagonal or $n=2$).
 		Note that in the assumption of the lemma the eigenvalues have
 		the order $\lambda_1\leq \lambda_2\leq \cdots \leq \lambda_n$.
 		It is  well known that, for a Hermitian matrix,
 		any diagonal element is   less than or equals to   the  largest eigenvalue.
 		In particular,
 		\begin{equation}
 			\label{largest-eigen1}
 			\lambda_n \geq \mathrm{{\bf a}}.
 		\end{equation}
 		
 		It only requires to prove \eqref {meishi}, since  \eqref{mei-23-shi} is a consequence of  \eqref{meishi}, \eqref{largest-eigen1}  and
 		\begin{equation}
 			\label{trace}
 			\sum_{i=1}^{n}\lambda_i=\mbox{tr}(\mathrm{A})=\sum_{\alpha=1}^{n-1} d_{\alpha}+\mathrm{{\bf a}}.
 		\end{equation}

 		Let's denote   $I=\{1,2,\cdots, n-1\}$. We divide the index set   $I$ into two subsets  by
 		$${\bf B}=\{\alpha\in I: |\lambda_{\alpha}-d_{i}|\geq \epsilon, \mbox{   }\forall i\in I\} $$
 		and $ {\bf G}=I\setminus {\bf B}=\{\alpha\in I: \mbox{There exists an $i\in I$ such that }
 		|\lambda_{\alpha}-d_{i}| <\epsilon\}.$
 		
 		To complete the proof we need to prove ${\bf G}=I$ or equivalently ${\bf B}=\emptyset$.
 		It is easy to see that  for any $\alpha\in {\bf G}$, one has
 		\begin{equation}
 			\label{yuan-lemma-proof1}
 			\begin{aligned}
 				|\lambda_\alpha|< \sum_{i=1}^{n-1}|d_i| + \epsilon.
 			\end{aligned}
 		\end{equation}
 		
 		Fix $ \alpha\in {\bf B}$,  we are going to 
 		estimate $\lambda_\alpha$.
 		The eigenvalue $\lambda_\alpha$ satisfies
 		\begin{equation}
 			\label{characteristicpolynomial}
 			\begin{aligned}
 				(\lambda_{\alpha} -\mathrm{{\bf a}})\prod_{i=1}^{n-1} (\lambda_{\alpha}-d_i)
 				= \sum_{i=1}^{n-1} (|a_{i}|^2 \prod_{j\neq i} (\lambda_{\alpha}-d_{j})).
 			\end{aligned}
 		\end{equation}
 		By the definition of ${\bf B}$, for  $\alpha\in {\bf B}$, one then has $|\lambda_{\alpha}-d_i|\geq \epsilon$ for any $i\in I$.
 		We therefore derive
 		\begin{equation}
 			\begin{aligned}
 				|\lambda_{\alpha}-\mathrm{{\bf a}} |=  \left|\sum_{i=1}^{n-1} \frac{|a_i|^2}{\lambda_{\alpha}-d_{i}}\right|\leq\sum_{i=1}^{n-1} \frac{|a_i|^2}{|\lambda_{\alpha}-d_{i}|}\leq
 				\frac{1}{\epsilon}\sum_{i=1}^{n-1} |a_i|^2, \mbox{ if } \alpha\in {\bf B}.
 			\end{aligned}
 		\end{equation}
 		Hence,  for $\alpha\in {\bf B}$, we obtain
 		\begin{equation}
 			\label{yuan-lemma-proof2}
 			\begin{aligned}
 				\lambda_\alpha \geq \mathrm{{\bf a}}-\frac{1}{\epsilon}\sum_{i=1}^{n-1} |a_i|^2.
 			\end{aligned}
 		\end{equation}

 		We shall use proof by contradiction to prove  ${\bf B}=\emptyset$.
 		For a set ${\bf S}$, we denote $|{\bf S}|$ the  cardinality of ${\bf S}$.
 		Assume ${\bf B}\neq \emptyset$.
 		Then $|{\bf B}|\geq 1$, and so $|{\bf G}|=n-1-|{\bf B}|\leq n-2$. 

 		In the case of ${\bf G}\neq \emptyset$, we compute the trace of the matrix $A$ as follows:
 		\begin{equation}
 			\begin{aligned}
 				\mbox{tr}(\mathrm{A})=\,&
 				\lambda_n+
 				\sum_{\alpha\in {\bf B}}\lambda_{\alpha} + \sum_{\alpha\in  {\bf G}}\lambda_{\alpha}\\
 				> \,&
 				\lambda_n+
 				|{\bf B}| (\mathrm{{\bf a}}-\frac{1}{\epsilon}\sum_{i=1}^{n-1} |a_i|^2 )-|{\bf G}| (\sum_{i=1}^{n-1}|d_i|+\epsilon ) \\
 				\geq \,&
 				2\mathrm{{\bf a}}-\frac{1}{\epsilon}\sum_{i=1}^{n-1} |a_i|^2 -(n-2) (\sum_{i=1}^{n-1}|d_i|+\epsilon )
 				\\
 				\geq \,& \sum_{i=1}^{n-1}d_i +\mathrm{{\bf a}}= \mbox{tr}(\mathrm{A}),
 			\end{aligned}
 		\end{equation}
 		where we use  \eqref{guanjian2},   \eqref{largest-eigen1}, \eqref{yuan-lemma-proof1} and \eqref{yuan-lemma-proof2}.
 		This is a contradiction.
 		
 		In the case of ${\bf G}=\emptyset$, one knows that
 		\begin{equation}
 			\begin{aligned}
 				\mbox{tr}(\mathrm{A})
 				\geq
 				\mathrm{{\bf a}}+
 				(n-1) (\mathrm{{\bf a}}-\frac{1}{\epsilon}\sum_{i=1}^{n-1} |a_i|^2 )
 				>   \sum_{i=1}^{n-1}d_i +\mathrm{{\bf a}}= \mbox{tr}(\mathrm{A}).
 			\end{aligned}
 		\end{equation}
 		Again, it is a contradiction.  Thus  ${\bf B}=\emptyset$ as required. 
 		
 	\end{proof}
 	
 	We apply Lemma \ref{refinement} to prove Lemma \ref{yuan's-quantitative-lemma} via a deformation argument.

 	\begin{proof}
 		[Proof of Lemma \ref{yuan's-quantitative-lemma}]
 		Without loss of generality,  we assume $n\geq 3$ and  $\sum_{i=1}^{n-1} |a_i|^2>0$
 		(otherwise  $n=2$ or the matrix $\mathrm{A}$ is diagonal, and then we are done).
 		Fix $a_1, \cdots, a_{n-1}$,
 		$d_1, \cdots, d_{n-1}$. 
 		Denote $\lambda_1(\mathrm{{\bf a}}), \cdots, \lambda_n(\mathrm{{\bf a}})$ by
 		the eigenvalues of $\mathrm{A}$ with
 		the order  $\lambda_1(\mathrm{{\bf a}})\leq \cdots\leq \lambda_n(\mathrm{{\bf a}})$. 
 		Clearly,  the eigenvalues $\lambda_i(\mathrm{{\bf a}})$ are all continuous functions 
 		in $\mathrm{{\bf a}}$.
 		For simplicity, we write $\lambda_i=\lambda_i(\mathrm{{\bf a}})$. 

 		Fix $\epsilon>0$.  Let $I'_\alpha=(d_\alpha-\frac{\epsilon}{2n-3}, d_\alpha+\frac{\epsilon}{2n-3})$ and
 		$$P_0'=\frac{2n-3}{\epsilon}\sum_{i=1}^{n-1} |a_i|^2+ (n-1)\sum_{i=1}^{n-1} |d_i|+ \frac{(n-2)\epsilon}{2n-3}.$$
 		In what follows we assume  $\mathrm{{\bf a}}\geq P_0'$ (i.e. \eqref{guanjian1-yuan} holds).
 		The connected components of $\bigcup_{\alpha=1}^{n-1} I_{\alpha}'$ are as in the following:
 		$$J_{1}=\bigcup_{\alpha=1}^{j_1} I_\alpha',
 		J_2=\bigcup_{\alpha=j_1+1}^{j_2} I_\alpha'  \cdots, J_i =\bigcup_{\alpha=j_{i-1}+1}^{j_i} I_\alpha', \cdots, 
 		J_{m} =\bigcup_{\alpha=j_{m-1}+1}^{n-1} I_\alpha'.$$
 		Moreover
 		\begin{equation}
 			\begin{aligned}
 				J_i\bigcap J_k=\emptyset, \mbox{ for }   1\leq i<k\leq m. \nonumber
 			\end{aligned}
 		\end{equation}
 		
 		Let  $$ \mathrm{{\bf \widetilde{Card}}}_k:[P_0',+\infty)\rightarrow \mathbb{N}$$
 		be the function that counts the eigenvalues which lie in $J_k$.
 		(Note that when the eigenvalues are not distinct,  the function $\mathrm{{\bf \widetilde{Card}}}_k$ denotes  the summation of all the algebraic  multiplicities of  distinct eigenvalues which
 		lie in $J_k$).
 		This function measures the number of the  eigenvalues which lie in $J_k$.
 		
 		The crucial ingredient is that  Lemma \ref{refinement}  yields the continuity of   $\mathrm{{\bf \widetilde{Card}}}_i(\mathrm{{\bf a}})$ for $\mathrm{{\bf a}}\geq P_0'$. More explicitly,
 		by using  Lemma \ref{refinement} and  $$\lambda_n \geq {\bf a}\geq P_0'>\sum_{i=1}^{n-1}|d_i|+\frac{\epsilon}{2n-3}$$ we conclude that
 		if   $\mathrm{{\bf a}}$ satisfies the quadratic growth condition \eqref{guanjian1-yuan} then
 		\begin{equation}
 			\label{yuan-lemma-proof5}
 			\begin{aligned}
 				\,& \lambda_n \in \mathbb{R}\setminus (\bigcup_{k=1}^{n-1} \overline{I_k'})
 				=\mathbb{R}\setminus (\bigcup_{i=1}^m \overline{J_i}), \\
 				\,& \lambda_\alpha \in \bigcup_{i=1}^{n-1} I_{i}'=\bigcup_{i=1}^m J_{i} \mbox{ for } 1\leq\alpha\leq n-1.
 			\end{aligned}
 		\end{equation}
 		Hence,  $\mathrm{{\bf \widetilde{Card}}}_i(\mathrm{{\bf a}})$ is a continuous function
 		in the variable $\mathrm{{\bf a}}$. So it is a constant.
 		Together with  the line of the proof   of \cite[Lemma 1.2]{CNS3}
 		we see
 		that $ \mathrm{{\bf \widetilde{Card}}}_i(\mathrm{{\bf a}}) =j_i-j_{i-1}$ for sufficiently large $\mathrm{{\bf a}}$.
 		Here we denote $j_0=0$ and $j_m=n-1$.
 		The constant of $ \mathrm{{\bf \widetilde{Card}}}_i$  therefore follows that
 		$$ \mathrm{{\bf \widetilde{Card}}}_i(\mathrm{{\bf a}})
 		=j_i-j_{i-1}.$$
 		We thus know that the   $(j_i-j_{i-1})$ eigenvalues
 		$$\lambda_{j_{i-1}+1}, \lambda_{j_{i-1}+2}, \cdots, \lambda_{j_i}$$
 		lie in the connected component $J_{i}$.
 		Thus, for any $j_{i-1}+1\leq \gamma \leq j_i$,  we have $I_\gamma'\subset J_i$ and  $\lambda_\gamma$
 		lies in the connected component $J_{i}$.
 		Therefore,
 		$$|\lambda_\gamma-d_\gamma| < \frac{(2(j_i-j_{i-1})-1) \epsilon}{2n-3}\leq \epsilon.$$
 		Here we also use the fact that $d_\gamma$ is midpoint of  $I_\gamma'$ and 
 		every $J_i\subset \mathbb{R}$ is an open subset.
 		
 		To be brief,  if for fixed index $1\leq i\leq n-1$ the eigenvalue $\lambda_i(P_0')$ lies in $J_{\alpha}$ for some $\alpha$, 
 		then  Lemma \ref{refinement} implies that, for any ${\bf a}>P_0'$, the corresponding eigenvalue  $\lambda_i({\bf a})$ lies in the same  interval $J_{\alpha}$.
 		The computation of $\mathrm{{\bf \widetilde{Card}}}_k$ can be done by letting $\mathrm{\bf a}\rightarrow+\infty$.
 		
 		
 	\end{proof}

 \end{appendix}

\subsection*{Acknowledgements} 
The author was supported by the National Natural Science of Foundation of China, Grant No. 11801587.

\end{document}